\documentclass[12pt]{article}
\title{Topologizing interpretable groups in $p$-adically closed fields}
\author{Will Johnson}

\usepackage{amsmath, amssymb, amsthm}    	
\usepackage{fullpage} 	
\usepackage{amscd}
\usepackage{hyperref}
\usepackage[all]{xy}
\usepackage[T1]{fontenc}
\usepackage{lmodern}
\usepackage{centernot}
\usepackage{enumitem}

\DeclareMathOperator*{\ind}{\raise0.2ex\hbox{\ooalign{\hidewidth$\vert$\hidewidth\cr\raise-0.9ex\hbox{$\smile$}}}}

\newcommand{\dimind}{\ind^{\dim}}
\newcommand{\pCF}{p\mathrm{CF}}

\newcommand{\Aut}{\operatorname{Aut}}

\newcommand{\id}{\operatorname{id}}

\newcommand{\acl}{\operatorname{acl}}
\newcommand{\dcl}{\operatorname{dcl}}
\newcommand{\tp}{\operatorname{tp}}

\newcommand{\bd}{\operatorname{bd}}

\newtheorem{theorem}{Theorem}[section] 

\newtheorem{lemma}[theorem]{Lemma}

\newtheorem{corollary}[theorem]{Corollary}
\newtheorem{fact}[theorem]{Fact}

\newtheorem{conjecture}[theorem]{Conjecture}

\newtheorem{proposition}[theorem]{Proposition}
\newtheorem{proposition-eh}[theorem]{Proposition(?)}
\newtheorem*{theorem-star}{Theorem}
\newtheorem*{conjecture-star}{Conjecture}
\newtheorem*{lemma-star}{Lemma}
\newtheorem{claim}[theorem]{Claim}

\theoremstyle{definition}
\newtheorem{definition}[theorem]{Definition}
\newtheorem{example}[theorem]{Example}

\newtheorem{remark}[theorem]{Remark}

\newtheorem*{warning}{Warning}

\theoremstyle{remark}

\newtheorem*{acknowledgment}{Acknowledgments}

\newcommand{\Qq}{\mathbb{Q}}
\newcommand{\eq}{\mathrm{eq}}

\newcommand{\Rr}{\mathbb{R}}

\newcommand{\Mm}{\mathbb{M}}

\newcommand{\Oo}{\mathcal{O}}

\newenvironment{claimproof}[1][\proofname]
               {
                 \proof[#1]
                 
               }
               {
                 \endproof
               }

\begin{document}

\maketitle

\begin{abstract}
  We consider interpretable topological spaces and topological groups
  in a $p$-adically closed field $K$.  We identify a special class of
  ``admissible topologies'' with topological tameness properties like
  generic continuity, similar to the topology on definable subsets of
  $K^n$.  We show every interpretable set has at least one admissible
  topology, and every interpretable group has a unique admissible
  group topology.  We then consider definable compactness (in the
  sense of Fornasiero) on interpretable groups.  We show that an
  interpretable group is definably compact if and only if it has
  finitely satisfiable generics (\textit{fsg}), generalizing an
  earlier result on definable groups.  As a consequence, we see that
  \textit{fsg} is a definable property in definable families of
  interpretable groups, and that any \textit{fsg} interpretable group
  defined over $\Qq_p$ is definably isomorphic to a definable group.
\end{abstract}

\section{Introduction}
The theory of $p$-adically closed
fields, denoted $\pCF$, has a cell decomposition theorem \cite[\S4,
  Theorem~$1.1'$]{vdDS} analogous to the cell decomposition in
o-minimal theories.  This in turn yields a dimension theory
\cite[\S3]{vdDS} and topological tameness results.  Here are three
representative results from the topological tameness of $\pCF$:
\begin{itemize}
\item If $f : X \to Y$ is a definable function, then $f$ is continuous
  on a dense open subset of $X$ \cite[\S4, Theorem~$1.1'$]{vdDS}.
\item If $X$ is definable and non-empty, then the frontier $\partial
  X$ has dimension strictly less than the dimension of $X$
  \cite[Theorem~3.5]{p-minimal-cells}.
\item If $X$ is definable of dimension $n$, then there is an open
  definable subset $X' \subseteq X$ such that $X'$ is an
  $n$-dimensional definable manifold and $\dim(X \setminus X') < n$
  (see Remark~\ref{definable-to-manifold}).
\end{itemize}
These results in turn allow one to give any definable group $G$ the
structure of a definable manifold in a canonical way \cite{Pillay-G-in-p}.

Unlike a typical o-minimal theory, $\pCF$ does \emph{not} have
elimination of imaginaries.  Consequently, interpretable sets and groups no longer
come with obvious topologies.  This is a shame, as there are some
important interpretable sets in $\pCF$, like the value group and the
set of balls.  Moreover, interpretable groups arise naturally in the
study of definable groups when forming quotient groups.

One can define a general class of \emph{interpretable topological
  spaces}, that is, definable topological spaces in
$p\mathrm{CF}^\eq$.  However, topological tameness results no longer
hold on this general class.  For example, if $X$ and $Y$ are the home
sort $K$ with the standard topology and discrete topology,
respectively, then the identity map $X \to Y$ is not generically
continuous.  We need to exclude things like the discrete topology on the
home sort.

In this paper, we define a special class of \emph{admissible}
interpretable topological spaces (Definition~\ref{adm-def}).  The
class of admissible topologies excludes cases like the discrete
topology on the home sort.  The class of admissible topologies has
many nice properties.  First, there are ``enough'' admissible
topologies:
\begin{enumerate}
	\item Every interpretable set has at least one admissible topology
	(Theorem~\ref{construction-1}).  Every interpretable group has a
	unique admissible group topology (Theorem~\ref{adm-group-thm}).
\item A definable set $D \subseteq K^n$ is admissible, as a
  topological subspace of $K^n$.
\item A definable manifold is admissible
  (Example~\ref{upsilon}).
  \item An interpretable subspace of an admissible topological space is
  admissible (Proposition~\ref{subspaces}).
\item A product or disjoint union of two admissible topological spaces
  is admissible (Proposition~\ref{products-etc}).
  \item Admissibility is preserved under interpretable homeomorphism.
\end{enumerate}
Second, admissible topologies are ``tame'' or nice:
\begin{enumerate}[resume]
\item Admissible topological spaces are Hausdorff.
\item If $X$ is admissible and $p \in X$, then some neighborhood of
  $p$ is interpretably homeomorphic to a definable subset of $K^n$.
\item If $X$ is admissible and $D \subseteq X$ is interpretable, then
  the frontier $\partial D$ has lower dimension than $D$
  (Proposition~\ref{frontier-dimension}).  See Section~\ref{sec-dim}
  for a review of dimension theory on interpretable sets.
\item If $f : X \to Y$ is interpretable, then there is an
  interpretable closed subset $D \subseteq X$ of lower dimension than $X$,
  such that $f$ is continuous on $X \setminus D$ (Proposition~\ref{gen-con}).
\item If $X$ is admissible, then there is an interpretable closed
  subset $D \subseteq X$ of lower dimension than $X$, such that $X
  \setminus D$ is everywhere locally homeomorphic to $K^n$, for $n =
  \dim(X)$ (Corollary~\ref{generically-n-manifold}).
\end{enumerate}
\begin{remark}
  One interesting corollary of the final point is that if $S$ is
  interpretable of dimension $n$, then there is an interpretable
  injection from an $n$-dimensional ball into $S$.  (This can also be
  proved directly.)
\end{remark}

Admissible interpretable topologies were introduced in
\cite[Section~4]{wj-o-minimal} in the esoteric context of o-minimal
theories without elimination of imaginaries.  Luckily, the arguments
of \cite{wj-o-minimal} carry over to the $p$-adic context with almost
no changes.  Nevertheless, we take the opportunity to clean up some of
the proofs and results.  We also slightly strengthen the definition
of ``admissible,'' in order to get cleaner theorems.

We apply the theory of admissibility to interpretable groups in
$\pCF$.  Classically, Pillay constructed a definable
manifold structure on any definable group in $\pCF$
\cite{Pillay-G-in-p}.  Admissibility allows us to run the same
arguments on interpretable groups, yielding the following:
\begin{theorem}[{= Theorem~\ref{adm-group-thm}}]
  \label{group-theorem}
  If $G$ is an interpretable group, then there is a unique admissible
  group topology on $G$.
\end{theorem}
When $G$ is definable, the admissible group topology is Pillay's
definable manifold structure on $G$.
Theorem~\ref{group-theorem}
feels counterintuitive in the case of the value group
$(\Gamma,+)$, which doesn't admit any obvious topology.  In fact, we
get the discrete topology:
\begin{proposition}[{= Remark~\ref{local-dim-remark}}]
  If $G$ is interpretable, the admissible group topology on $G$ is
  discrete if and only if $\dim(G) = 0$.  In particular, the
  admissible group topology on $\Gamma$ is the discrete topology.
\end{proposition}
Thus, nothing very interesting happens on 0-dimensional groups.
Nevertheless, it is convenient to have a canonical topology which
works uniformly across both definable groups and 0-dimensional
interpretable groups.  The admissible group topology is well-behaved
in several ways:
\begin{theorem}
  Let $G$ be an interpretable group with its admissible group topology, and
  let $H$ be an interpretable subgroup.
  \begin{enumerate}
  \item $H$ is always closed (Proposition~\ref{closed-subgroup}(\ref{cs1})).  $H$ is clopen if and only if $\dim(H) = \dim(G)$ (Proposition~\ref{open-subgroup}).
  \item The admissible group topology on $H$ is the subspace topology
    (Proposition~\ref{closed-subgroup}(\ref{cs2})).
  \item If $H$ is normal, then the admissible group topology on $G/H$
    is the quotient topology (Proposition~\ref{quotient-prop}(\ref{qp3})).
  \end{enumerate}
\end{theorem}
\begin{theorem}
  Let $f : G \to H$ be an interpretable homomorphism.
  \begin{enumerate}
  \item $f$ is continuous with respect to the admissible group
    topologies on $G$ and $H$ (Proposition~\ref{hom-cts}).
  \item If $f$ is injective, then $f$ is a closed embedding (Corollary~\ref{ce-cor}).
  \item If $f$ is surjective, then $f$ is an open map
    (Corollary~\ref{o-cor}).
  \end{enumerate}
\end{theorem}
\subsection{Application to \textit{fsg} groups} \label{intro-fsg}
In future work with Yao \cite{jy-abelian},
Theorem~\ref{group-theorem} will be used to generalize some of the
results of \cite{johnson-yao} to interpretable groups.  In the present
paper, we apply Theorem~\ref{group-theorem} to the study of
interpretable groups with \textit{fsg}.

Recall that an interpretable group has \emph{finitely satisfiable
generics} (\textit{fsg}) if there is a small model $M_0$ and a global
type $p \in S_G(\Mm)$ such that every left translate $g \cdot p$ is
finitely satisfiable in $M_0$.  This notion is due to Hrushovski,
Peterzil, and Pillay \cite{goo}, who show that generic sets behave
well in \textit{fsg} groups.  An interpretable subset $X \subseteq G$
is said to be \emph{left generic} or \emph{right generic} if $G$ can
be covered by finitely many left translates or right translates of
$X$, respectively.
\begin{fact}[{\cite[Proposition~4.2]{goo}}]\label{hpp-fact}
  Suppose $G$ has \textit{fsg}, witnessed by $p$ and $M_0$.
  \begin{enumerate}
  \item A definable set $X \subseteq G$ is left generic iff it is
    right generic.
  \item Non-generic sets form an ideal: if $X \cup Y$ is generic, then
    $X$ is generic or $Y$ is generic.
  \item A definable set $X$ is generic if and only if every left
    translate of $X$ intersects $G(M_0)$.
  \end{enumerate}
\end{fact}
The significance of \textit{fsg} is that it corresponds to
``definable compactness'' in several settings.  For example, if $G$ is
a group definable in a nice o-minimal structure, then $G$ has
\textit{fsg} if and only if $G$ is definably compact
\cite[Remark~5.3]{udi-anand}.  In $\pCF$, a definable group $G$ has
\textit{fsg} if and only if it is definably compact
\cite{O-P,johnson-fsg}.  With admissible group topologies in hand, the
same arguments generalize to interpretable groups:
\begin{theorem}[{= Theorem~\ref{fsg-char}}] \label{t1.6}
  An interpretable group $G$ has \textit{fsg} if and only if it is
  definably compact with respect to the admissible group topology.
\end{theorem}
Here, definable compactness is in the sense of Fornasiero \cite{fornasiero};
see Definition~\ref{fornasiero-definition}.
Using this, we obtain some consequences which have nothing to do with
topology:
\begin{theorem}
  \begin{enumerate}
  \item The \textit{fsg} property is definable in families: if
    $\{G_a\}_{a \in X}$ is an interpretable family of interpretable
    groups, and $X_{fsg}$ is the set of $a \in X$ such that $G_a$ has
    \textit{fsg}, then $X_{fsg}$ is interpretable (Corollary~\ref{fsg-def}).
  \item If $G$ is interpretable over $\Qq_p$ and $G$ has \textit{fsg},
    then $G$ is isomorphic to a definable group (Corollary~\ref{fsg-int-qp}).
  \end{enumerate}
\end{theorem}
Theorem~\ref{t1.6} says something more concrete for 0-dimensional
interpretable groups, such as groups interpretable in the value group
$\Gamma$.  To explain, we need a few preliminary remarks.
The structure $\Qq_p^\eq$ eliminates $\exists^\infty$ (even though the
theory $p\mathrm{CF}^\eq$ \emph{does not}).  Consequently, one can
define a class of ``pseudofinite'' interpretable sets, characterized
by the two properties:
\begin{itemize}
\item Over $\Qq_p$, pseudofiniteness agrees with finiteness.
\item Pseudofiniteness is definable in families.
\end{itemize}
See Proposition~\ref{psf-prop} for a precise formulation.
Pseudofiniteness can also be defined explicitly; see
Definition~\ref{psf-def} and Proposition~\ref{other-psf}.  For
example, $S$ is pseudofinite if and only if $S$ is definably compact
with respect to the discrete topology.  With pseudofiniteness in hand,
Theorem~\ref{t1.6} yields the following for 0-dimensional
interpretable groups:
\begin{theorem}[{= Proposition~\ref{p8.8}}]
  Let $G$ be a 0-dimensional interpretable group.  Then $G$ has
  \textit{fsg} if and only if $G$ is pseudofinite.
\end{theorem}
\subsection{Relation to prior work}
This paper leans heavily on \cite{wj-o-minimal}, \cite{Pillay-G-in-p}, and
\cite{johnson-fsg}.  On some level, this paper is merely the following
three observations:
\begin{enumerate}
\item The theory of ``admissible topologies'' from Sections~3 and 4 of
  \cite{wj-o-minimal} can be transfered from the o-minimal setting to
  $\pCF$.
\item Pillay's construction of definable manifold structures on
  definable groups \cite{Pillay-G-in-p} can then be applied to interpretable
  groups, giving a unique admissible group topology on any
  interpretable group.
\item The proof in \cite{johnson-fsg} that ``\textit{fsg} = definable
  compactness'' then generalizes to interpretable groups.
\end{enumerate}

In the o-minimal context, a similar idea of constructing a topology on
interpretable groups appears in the work of Eleftheriou, Peterzil, and
Ramakrishnan \cite[Theorem~8.7]{interpretable-groups}.  Using their
topology, the authors show that interpretable groups are definable
\cite[Theorem~8.22]{interpretable-groups}, which implies in hindsight
that the topology is Pillay's topology on definable groups.  (This
contrasts with $\pCF$, where there are non-definable interpretable
groups such as the value group and residue field.)

It is known that $\pCF$ eliminates imaginaries after adding the
so-called \emph{geometric sorts} to the language
\cite[Theorem~1.1]{pcf-ei}.  The $n$th geometric sort is a quotient of
$GL_n(K)$ and can be understood as a space of lattices in $K^n$.  It
is probably possible to circumvent much of this paper, especially
Section~\ref{adm-sec-1}, through the explicit description of
imaginaries, as explained in Section~\ref{geometric-ei} below.
However, the advantage of the present approach is that it is much more
likely to generalize to other settings, such as $P$-minimal theories,
where an explicit description of imaginaries is unknown.

\subsection{Outline}
In Section~\ref{sec-dim} we review the dimension theory for
imaginaries in geometric structures following
\cite{gagelman}.\footnote{This is necessary because the use of
\th-independence and $U^{\text{\th}}$-rank in \cite{wj-o-minimal} no
longer works in the non-rosy theory $\pCF$.  Note that the use of
\th-independence and $U^{\text{\th}}$-rank in \cite{wj-o-minimal} was
overkill; the dimension theory on imaginaries in geometric structures
from \cite{gagelman} would have sufficed.  In fact, in o-minimal
theories, the dimension on imaginaries agrees with
$U^{\text{\th}}$-rank.} In Section~\ref{random-review} we review the
notion of definable and interpretable topological spaces and
Fornasiero's definition of definable compactness \cite{fornasiero}.
In Section~\ref{adm-sec-1} we develop the theory of admissible
topologies in $\pCF$, showing that they satisfy topological tameness
properties akin to definable sets (\S\ref{ss-tame}), that there are
enough of them (\S\ref{ss-construct}), and that natural operations on
topological spaces preserve admissibility (\S\ref{closure-props}).  In
Section~\ref{adm-sec-2} we turn to admissible group topologies,
showing that each interpretable group has a unique admissible group
topology, and checking the topological properties of homomorphisms.
In Section~\ref{def-com-sec} we make a few remarks about definable
compactness in admissible interpretable topological spaces.  In
Section~\ref{fsg-sec} we apply this to the study of \textit{fsg}
interpretable groups, and in Section~\ref{0-dim} we analyze what
happens for 0-dimensional groups.  Finally, in Section~\ref{s-fd} we
consider future research directions, including possible
generalizations and open problems.

\subsection{Conventions}
\emph{Open maps are assumed to be continuous}.  If $X$ is a subset of
a topological space, then $\partial X$ denotes the frontier of $X$ and
$\bd(X)$ denotes the boundary.  If $E$ is an equivalence relation on a
set $X$ and $X'$ is a subset, then $E \restriction X'$ denotes the
restriction of $E$ to $X'$.  If $E$ is an equivalence relation on a
topological space $X$, then $X/E$ denotes the quotient topological
space.  When $X' \subseteq X$, we sometimes abbreviate $X'/(E
\restriction X')$ as $X'/E$.

Following \cite{panorama}, we define a \emph{$p$-adically closed
field} to be a field elementarily equivalent to $\Qq_p$, and we denote
the theory of $p$-adically closed fields by $\pCF$.  The term
``$p$-adically closed field'' is often used in a more general sense to
refer to any field elementarily equivalent to a finite extension of
$\Qq_p$.  For simplicity, we will not consider this more general
context.  However, \emph{all the results in this paper generalize to
$p$-adically closed fields in the broad sense}, replacing $\Qq_p$ with
its finite extensions in certain places.

Symbols like $x,y,z,a,b,c,\ldots$ can denote singletons or tuples.
Tuples are finite by default.  Letters $A,B,C,\ldots$ are usually
reserved for small sets of parameters, and letters $M, N$ are usually
reserved for small models.  We denote the monster model by $\Mm$.  We
maintain the distinction between definable and interpretable sets, as
well as the distinction between reals (in $\Mm$) and imaginaries (in
$\Mm^\eq$).  ``Definable'' means ``definable with parameters,'' and
``0-definable'' means ``definable without parameters.''  If $D$ is
definable, then $\ulcorner D \urcorner$ denotes ``the'' code of $D$,
which is well-defined up to interdefinability, but is usually
imaginary.  We write the $\acl$-dimension in geometric theories as
$\dim(a/B)$ for complete types and $\dim(X)$ for definable sets.

\section{Dimension theory in geometric structures} \label{sec-dim}
Let $\Mm$ be a monster model of a complete one-sorted theory $T$.  Recall the
following definitions from \cite{geostructs, gagelman}.
\begin{definition}
  $T$ is a \emph{pregeometric theory} if $\acl(-)$ satisfies the
  Steinitz exchange property.  $T$ is a \emph{geometric theory} if it
  is pregeometric and $\exists^\infty$ is eliminated.
\end{definition}
There is a well-known dimension theory on pregeometric structures.
This dimension theory assigns a dimension $\dim(a/B)$ to each complete
type $\tp(a/B)$ and a dimension $\dim(X)$ to each definable set $X$.
By work of Gagelman \cite{gagelman}, the dimension theory extends to
$T^{\eq}$.  We review this theory below.

\textbf{For the rest of this section, we assume $T$ is pregeometric.}

\begin{definition} \label{dim-def}
  Let $A \subseteq \Mm$ be a small set of parameters and $b =
  (b_1,\ldots,b_n)$ be a tuple in $\Mm$.  The tuple $b$ is
  \emph{$\acl$-independent (over $A$)} if $b_i \notin \acl(A \cup
  \{b_j : j \ne i\})$ for $1 \le i \le n$.  The \emph{dimension} of
  $b$ over $A$, written $\dim(b/A)$, is the length of a maximal
  subtuple $c$ of $b$ such that $c$ is $\acl$-independent over $A$.
  This is independent of the choice of $c$, assuming the exchange
  property.
\end{definition}
The following properties of dimension are well-known:
\begin{enumerate}
\item Automorphism invariance: if $\sigma \in \Aut(\Mm)$, then
  $\dim(\sigma(a)/\sigma(B)) = \dim(a/B)$.
\item Extension: given $a$ and $B \subseteq C$, there is $a' \equiv_B
  a$ with $\dim(a'/C) = \dim(a'/B) = \dim(a/B)$.
\item Additivity: $\dim(a,b/C) = \dim(a/Cb) + \dim(b/C)$.
\item Base monotonicity: if $B \subseteq C$, then $\dim(a/B) \ge \dim(a/C)$.
\item Finite character: Given $a, B$ there is a finite subset $B'
  \subseteq B$ witih $\dim(a/B') = \dim(a/B)$.
\item Anti-reflexivity: $\dim(a/B) = 0 \iff a \in \acl(B)$.
\end{enumerate}
The exchange property is perserved when imaginaries are named as
parameters:
\begin{fact}[{\cite[Lemma 3.1]{gagelman}}] \label{gagelman-exchange}
  Suppose $A \subseteq \Mm^{\eq}$ is small and $b, c \in \Mm$.  Then
  \begin{equation}
    b \in \acl(Ac) \setminus \acl(A) \implies c \in \acl(Ab).
  \end{equation}
\end{fact}
Therefore Definition~\ref{dim-def} can be generalized to define
$\dim(b/A)$ for real $b \in \Mm^n$ and imaginary $A \subseteq
\Mm^\eq$.  The six properties listed above continue to hold.  Finally,
we consider the case where $b$ is imaginary:
\begin{definition} \label{dim-def-2}
  Let $b$ be a tuple of imaginaries and $A$ be a set of
  imaginaries.  Let $c$ be a real tuple such that $b \in
  \acl^\eq(Ac)$.  Define
  \begin{equation*}
    \dim(b/A) := \dim(c/A) - \dim(c/Ab).
  \end{equation*}
\end{definition}
\begin{fact} \phantomsection \label{dim-basics}
  \begin{enumerate}
  \item In Definition~\ref{dim-def-2}, $\dim(b/A)$ is well-defined,
    independent of the choice of $c$.
  \item When $b$ is a real tuple, Definition~\ref{dim-def-2} agrees
    with Definition~\ref{dim-def}.
  \item $\dim(-/-)$ satisfies automorphism invariance, extension,
    additivity, base monotonicity, and finite character.
  \item $\dim(-/-)$ satisfies half of anti-reflexivity:
    \begin{equation*}
      b \in \acl^\eq(A) \implies \dim(b/A) = 0.
    \end{equation*}
  \end{enumerate}
\end{fact}
Fact~\ref{dim-basics} is proved in \cite[Lemma~3.3, Proposition~3.4]{gagelman}.  Gagelman
omits the proof of Base Monotonicity, which is mildly subtle, so we
review the proof for completeness:
\begin{proof}[Proof (of base monotonicity)]
  Suppose $a \in \Mm^\eq$ and $B \subseteq C \in \Mm^\eq$.  Take a
  real tuple $d \in \Mm^n$ such that $a \in \acl^\eq(Bd)$.  By the
  extension property, we may move $d$ by an automorphism and arrange
  $\dim(d/Ba) = \dim(d/Ca)$.  Then
  \begin{equation*}
    \dim(a/B) = \dim(d/B) - \dim(d/Ba) \ge \dim(d/C) - \dim(d/Ca) = \dim(a/C)
  \end{equation*}
  by base monotonicity for real tuples.
\end{proof}
\begin{definition} \label{set-def}
  Let $A$ be a small subset of $\Mm^\eq$ and $X$ be an
  $A$-interpretable set.  Then $\dim(X) := \max_{c \in X}
  \dim(c/A)$.  When $X = \varnothing$, we define $\dim(X) =
  -\infty$.
\end{definition}
\begin{fact}[{\cite[p.\@ 320--321]{gagelman}}]
  \label{fact2.7}
  In Definition~\ref{set-def}, $\dim(X)$ does not depend on $A$.
\end{fact}
Fact~\ref{fact2.7} follows formally from base monotonicity and
extension.  Then Propositions~\ref{dim-of-sets} and \ref{fibers} below
follow formally from Fact~\ref{dim-basics} via the usual proofs.
\begin{proposition} \label{dim-of-sets}
  Let $X, Y$ be interpretable sets.
  \begin{enumerate}
  \item $\dim(X) \ge 0$ iff $X \ne \varnothing$.
  \item \label{dos2} If $X$ is finite, then $\dim(X) \le 0$.
  \item If $X$ is definable, then $\dim(X) \le 0$ if and only if $X$
    is finite.
  \item If $X, Y$ are in the same sort, then $\dim(X \cup Y) =
    \max(\dim(X),\dim(Y))$.
  \item $\dim(X \times Y) = \dim(X) + \dim(Y)$.
  \end{enumerate}
\end{proposition}
\begin{proposition} \label{fibers}
  Let $f : X \to Y$ be an interpretable function between two
  interpretable sets.
  \begin{enumerate}
  \item If every fiber has dimension at most $k$, then $\dim(X) \le k + \dim(Y)$.
  \item If every fiber has dimension at least $k$, then $\dim(X) \ge k + \dim(Y)$.
  \item If every fiber has dimension exactly $k$, then $\dim(X) = k +
    \dim(Y)$.
  \item If $f$ is surjective, then $\dim(X) \ge \dim(Y)$.
  \item If $f$ is injective, or more generally if $f$ has finite
    fibers, then $\dim(X) \le \dim(Y)$.
  \item If $f$ is a bijection, then $\dim(X) = \dim(Y)$.
  \end{enumerate}
\end{proposition}

Recall that the pregeometric theory $T$ is \emph{geometric} if it
eliminates $\exists^\infty$.
\begin{fact}[{\cite[Proposition 3.7]{gagelman}}] \phantomsection \label{continuity-01}
  \begin{enumerate}
  \item \label{uno} If $a \in \Mm^n$ and $B \subseteq \Mm^\eq$, then $\dim(a/B)$
    is the minimum of $\dim(X)$ as $X$ ranges over $B$-definable sets
    containing $a$.
  \item \label{dos} Suppose $T$ is geometric.  If $a \in \Mm^\eq$ and $B \subseteq
    \Mm^\eq$, then $\dim(a/B)$ is the minimum of $\dim(X)$ as $X$
    ranges over $B$-interpretable sets containing $a$.
  \end{enumerate}
\end{fact}
The assumption that $T$ is geometric is necessary in (\ref{dos}), as
shown by the following example.  Suppose $\Mm$ is an equivalence
relation with infinitely many equivalence classes of size $n$, for
each positive integer $n$.  By saturation, there are also infinite
equivalence classes.  The following facts are easy to verify:
\begin{enumerate}
\item For $A \subseteq \Mm$, the algebraic closure $\acl(A)
  \subseteq \Mm$ is the union of $A$ and all finite equivalence
  classes intersecting $A$.
\item $\acl$ satisfies exchange (on $\Mm$).
\item If $C$ is an equivalence class and $b = \ulcorner C \urcorner
  \in \Mm^\eq$ is its code, then
  \begin{equation*}
    \dim(b/\varnothing) = 
    \begin{cases}
      1 & \text{ if $C$ is finite} \\
      0 & \text{ if $C$ is infinite.}
    \end{cases}
  \end{equation*}
\item If $b$ is the code of an infinite equivalence class and $X$ is
  an $\varnothing$-interpretable set containing $b$, then $\dim(X) =
  1 > \dim(b/\varnothing) = 0$.
\end{enumerate}
Another special property of geometric theories is that dimension is
definable in families:
\begin{fact}[{\cite[Fact~2.4]{gagelman}}] \label{def-defable}
  Suppose $T$ is geometric.  Let $\{X_a\}_{a \in Y}$ be a definable
  family of definable sets.  Then for any $k$, the set $\{a \in Y :
  \dim(X_a) = k\}$ is definable.
\end{fact}
This holds for interpretable families as well:
\begin{proposition} \label{dim-definable-omega}
  Suppose $T$ is geometric.  Let $\{X_a\}_{a \in Y}$ be an
  interpretable family of interpretable sets.  Then for any $k$, the
  set $\{a \in Y : \dim(X_a) = k\}$ is interpretable.
\end{proposition}
\begin{proof}[Proof sketch]
  Let $X$ be an interpretable set, defined as $D/E$ for some definable
  set $D \subseteq \Mm^n$ and definable equivalence relation $E$ on
  $D$.  Let $[a]_E$ denote the $E$-equivalence class of $a \in D$.
  Let $D_j = \{a \in D : \dim([a]_E) = j\}$.  Let $X_j$ be the
  quotient $D_j/E$.  Each set $D_j$ is definable by
  Fact~\ref{def-defable}.  Using Propositions~\ref{dim-of-sets} and
  \ref{fibers}, one sees that
  \begin{equation*}
    \dim(X) = \dim\left(\bigcup_{j = 0}^n X_j\right) = \max_{0 \le j \le n} \dim(X_j) =
    \max_{0 \le j \le n} \left(\dim(D_j) - j\right).
  \end{equation*}
  This calculates $\dim(X)$ in a definable way.
\end{proof}

\subsection{Dimensional independence}
Continue to assume $T$ is pregeometric, but not necessarily geometric.
\begin{lemma} \label{acl-alt}
  $\dim(a/B) = \dim(a/\acl^\eq(B))$.
\end{lemma}
\begin{proof}
  Clear from the definitions.
\end{proof}
\begin{definition}\label{first-ind-def}
  Suppose $a, b \in \Mm$ and $C \subseteq \Mm^\eq$.  Then $a
  \dimind_C b$ means that $\dim(a,b/C) = \dim(a/C) + \dim(b/C)$.
\end{definition}
\begin{lemma} \label{ind-1}
  Suppose $a, b \in \Mm$ and $C \subseteq \Mm^\eq$.
  \begin{enumerate}
  \item \label{eins} $a \dimind_C b \iff b \dimind_C a$.
  \item \label{drei} $a \dimind_C b$ if and only if $\dim(a/C) = \dim(a/Cb)$.
  \item \label{vier} If $a' \in \acl^\eq(Ca)$ and $b' \in
    \acl^\eq(Cb)$, then $a \dimind_C b \implies a' \dimind_C b'$.
  \end{enumerate}
\end{lemma}
\begin{proof}
  (\ref{eins}) is trivial.  (\ref{drei}) holds because
  \begin{equation*}
    \dim(a/C) + \dim(b/C) - \dim(a,b/C) = \dim(a/C) - \dim(a/Cb) \ge 0
  \end{equation*}
  by additivity and base monotonicity.  For part (\ref{vier}), we may
  assume $a' = a$ by symmetry.  Then $a \dimind_C b \implies \dim(a/C) =
  \dim(a/Cb)$.  By Lemma~\ref{acl-alt}, this means
  \begin{equation*}
    \dim(a/\acl^\eq(C)) = \dim(a/\acl^\eq(Cb)).
  \end{equation*}
  But $\acl^\eq(C) \subseteq \acl^\eq(Cb') \subseteq \acl^\eq(Cb)$, so
  base monotonicity gives
  \begin{equation*}
    \dim(a/\acl^\eq(C)) = \dim(a/\acl^\eq(Cb')) = \dim(a/\acl^\eq(Cb)).
  \end{equation*}
  By Lemma~\ref{acl-alt} again, $\dim(a/C) = \dim(a/Cb')$, and so $a
  \dimind_C b'$.
\end{proof}
By Lemma~\ref{ind-1}(\ref{vier}), $a \dimind_C b$ depends only on $a$ and $b$ as
sets.
\begin{definition}
  If $A, B, C \subseteq \Mm^\eq$, then $A \dimind_C B$ means that $A_0
  \dimind_C B_0$ for all finite subsets $A_0 \subseteq A$ and $B_0
  \subseteq B$.
\end{definition}
This extends Definition~\ref{first-ind-def} by
Lemma~\ref{ind-1}(\ref{vier}).
\begin{proposition} \phantomsection \label{ind-2}
  \begin{enumerate}
  \item $A \dimind_C B \iff B \dimind_C A$.
  \item \label{silly-mon} If $A' \subseteq \acl^\eq(CA)$ and $B' \subseteq
    \acl^\eq(CB)$, then $A \dimind_C B \implies A' \dimind_C B'$.
  \item \label{sami} $a \dimind_C B$ holds iff $\dim(a/C) = \dim(a/CB)$.
  \item \label{otkhi} If $B_1 \subseteq B_2 \subseteq B_3$, then
    \begin{equation*}
      A \dimind_{B_1} B_3 \iff (A \dimind_{B_1} B_2 \text{ and } A \dimind_{B_2} B_3).
    \end{equation*}
  \end{enumerate}
\end{proposition}
\begin{proof}
  The first two points are clear from Lemma~\ref{ind-1}.  For part
  (\ref{sami}), use base monotonicity and finite character to reduce
  to the case where $B$ is finite, which is
  Lemma~\ref{ind-1}(\ref{drei}).  In part (\ref{otkhi}), we may assume
  $A$ is a finite tuple $a$, in which case (\ref{otkhi}) says
  \begin{equation*}
    \dim(a/B_1) = \dim(a/B_3) \iff (\dim(a/B_1) = \dim(a/B_2) \text{
      and } \dim(a/B_2) = \dim(a/B_3)).
  \end{equation*}
  This holds because $\dim(a/B_1) \ge \dim(a/B_2) \ge \dim(a/B_3)$ by
  base monotonicity.
\end{proof}

\begin{lemma} \label{type-ish}
  Let $a$ be a real tuple, possibly infinite.  Let $B \subseteq C$ be
  sets of imaginaries.  Then there is a consistent partial $\ast$-type
  $\Sigma_{a,B,C}(x)$ whose realizations are the tuples $a' \equiv_B
  a$ such that $a' \dimind_B C$.
\end{lemma}
\begin{proof}
  Let $I$ be a set indexing the tuple $a = (a_i : i \in I)$.  For each
  finite $J \subseteq I$, let $\pi_J$ be the projection map from
  $I$-tuples to $J$-tuples.  (For example, $\pi_J(a)$ is the finite
  subtuple of $a$ determined by $J$.)  If $a' \in \Mm^I$, note that
  \begin{itemize}
  \item $a' \equiv_B a$ holds iff $\pi_J(a') \equiv_B \pi_J(a)$ for
    all finite $J \subseteq I$.
  \item $a' \dimind_B C$ holds iff $\pi_J(a') \dimind_B C$ for
    all finite $J \subseteq I$, because of the definition of
    $\dimind$.
  \end{itemize}
  Therefore we may reduce to the case where $I$ is finite.  Let $n =
  \dim(a/B)$.  If $a' \equiv_B a$, then $\dim(a'/B) = n$ and so
  \begin{equation*}
    a' \dimind_B C \iff \dim(a'/B) \le \dim(a'/C) \iff n \le
    \dim(a'/C) \qquad \qquad \text{(for $a' \equiv_B a$)}.
  \end{equation*}
  Therefore
  \begin{equation*}
    \{a' \in \Mm^I : a' \equiv_B a, ~ a' \dimind_B C\} = \{a' \in
    \Mm^I : a' \equiv_B a, ~ \dim(a'/C) \ge n\}.
  \end{equation*}
  The condition $a' \equiv_B a$ is defined by the type $\tp(a/B)$.  By
  Fact~\ref{continuity-01}(\ref{uno}), the condition $\dim(a'/C) \ge n$ is
  defined by the type
  \begin{equation*}
    \{\neg \varphi(x,c) : \varphi \in L, ~ c \in C, ~ \dim(\varphi(\Mm,c)) < n\}.
  \end{equation*}
  Therefore $\{a' \in \Mm^I : a' \equiv_B a, ~ a' \dimind_B C\}$
  is type-definable.  Finally, the set is non-empty by the extension
  property of $\dim(-/-)$.
\end{proof}
\begin{proposition} \label{extension-2}
  Let $A, B, C$ be small subsets of $\Mm^\eq$.  Then there is $\sigma
  \in \Aut(\Mm/C)$ such that $\sigma(A) \dimind_C B$.
\end{proposition}
\begin{proof}
  Take an infinite real tuple $a$ such that $A \subseteq \dcl^\eq(a)$.
  By Lemma~\ref{type-ish}, there is some $a' \equiv_C a$ such that $a'
  \dimind_C B$.  Equivalently, there is $\sigma \in \Aut(\Mm/C)$ such
  that $\sigma(a) \dimind_C B$.  Now $\sigma(A) \subseteq
  \dcl^\eq(\sigma(a))$, so $\sigma(A) \dimind_C B$ holds by
  Proposition~\ref{ind-2}(\ref{silly-mon}).
\end{proof}

\section{Interpretable topological spaces and definable compactness} \label{random-review}
Let $M$ be any structure.
\begin{definition}
  A topology on an interpretable set $X$ is an \emph{interpretable
  topology} if there is an interpretable basis of opens, i.e., an
  interpretable family $\{S_a\}_{a \in Y}$ such that $\{S_a : a \in
  Y\}$ is a basis for the topology.  An \emph{interpretable
  topological space} is an interpretable set with an interpretable
  topology.
\end{definition}
For example, if $M \models p\mathrm{CF}$ then $M^n$ with the standard
topology is an interpretable topological space.

A family of sets $\mathcal{F}$ is \emph{downwards-directed} if for any
$X, Y \in \mathcal{F}$, there is $Z \in \mathcal{F}$ with $Z \subseteq
X \cap Y$.  In topology, compactness can be defines as follows: a
topological space $X$ is compact if any downwards-directed family of
non-empty closed subsets of $X$ has non-empty intersection.
\begin{definition} \label{fornasiero-definition}
  An interpretable topological space $X$ is \emph{definably compact}
  (in the sense of Fornasiero) if every downwards-directed
  interpretable family of non-empty closed subsets of $X$ has
  non-empty intersection.  An interpretable subset $D \subseteq X$ is
  \emph{definably compact} if it is definably compact with respect to
  the subspace topology.
\end{definition}
Definable compactness in this sense was studied independently by
Fornasiero \cite{fornasiero} and the author \cite{wj-o-minimal}.  Many of the expected
properties hold:
\begin{fact} \phantomsection \label{dc-fact}
  \begin{enumerate}
\item \label{df1} If $X$ is an interpretable topological space and $X$ is compact,
  then $X$ is definably compact (clear).
\item If $X, Y$ are definably compact, then the disjoint union $X
  \sqcup Y$ and the product $X \times Y$ are definably compact
  \cite[Lemma~3.5(2), Proposition~3.7]{wj-o-minimal}.
\item If $X$ is a Hausdorff interpretable topological space and $D
  \subseteq X$ is a definably compact interpretable subset, then $D$
  is closed \cite[Lemma~3.8]{wj-o-minimal}.
\item Finite sets are definably compact (clear).
\item If $X$ is an interpretable topological space and $D_1, D_2$ are
  definably compact interpretable subsets, then the union $D_1 \cup
  D_2$ is definably compact \cite[Lemma~3.5(2)]{wj-o-minimal}.
\item If $f : X \to Y$ is an interpretable continuous map and $X$ is
  definably compact, then the image $f(X)$ is definably compact as a
  subset of $Y$ \cite[Lemma~3.4]{wj-o-minimal}.
\item Definable compactness and non-compactness are preserved in
  elementary extensions (clear).
\end{enumerate}
\end{fact}
In $p$-adically closed fields definable compactness has additional
nice properties:
\begin{fact} \phantomsection \label{dc-fact-pcf}
  \begin{enumerate}
\item \label{dfp1} If $M$ is a $p$-adically closed field, a definable set $X
  \subseteq M^n$ is definably compact iff $X$ is closed and bounded
  \cite[Lemmas~2.4,2.5]{johnson-yao}.
\item Consequently, if $M = \Qq_p$, then a definable set $X \subseteq
  M^n$ is definably compact iff it is compact.
\end{enumerate}
\end{fact}

\section{Admissible topologies in $p$-adically closed fields} \label{adm-sec-1}
Work in a monster model $\Mm$ of $\pCF$.  We give $\Mm$ the standard
valuation topology and $\Mm^n$ the product topology.  If $X \subseteq
\Mm^n$ is definable, give $X$ the subspace topology.  This allows us
to regard any definable set as a definable topological space.
\begin{definition}
  Let $X$ be a Hausdorff interpretable topological space.
  \begin{enumerate}
  \item \label{ab1} $X$ is a \emph{definable manifold} if $X$ is covered by
    finitely many open subsets $U_1, \ldots, U_n$, such that $U_i$ is
    interpretably homeomorphic to an open definable subset of
    $\Mm^{k_i}$ for some $k_i$.
  \item \label{ab2} $X$ is \emph{locally Euclidean} if for every point $p \in X$,
    there is a neighborhood $U \ni p$ and an interpretable
    homeomorphism from $U$ to an open subset of $\Mm^n$ for some $n$
    depending on $p$.
  \item $X$ is \emph{locally definable} if for every point $p \in X$,
    there is a neighborhood $U \ni p$ and an interpretable
    homeomorphism from $U$ to a definable subspace of $\Mm^n$ for some
    $n$ depending on $p$.
  \end{enumerate}
\end{definition}
\begin{remark}
  \begin{enumerate}
  \item In (\ref{ab1}) and (\ref{ab2}), we allow the dimension to vary
    from point to point(!)
  \item Definable manifolds and locally Euclidean spaces are similar.
    The difference is that in a definable manfiold the atlas is
    finite, whereas in a locally Euclidean space the atlas can be
    infinite.
  \item Because we are working in a monster model, the atlas in local
    Euclideanity must be uniformly definable, i.e., the neighborhood
    $U \ni p$ and open embedding $U \to \Mm^n$ must have bounded
    complexity.  The same holds for local definability.
  \item $\Mm$ with the discrete topology is a locally Euclidean
    definable topological space that is not a definable manifold.
  \item If $X$ is a definable manifold, then $X$ is in interpretable
    bijection with a definable set.  So definable manifolds are
    essentially definable objects.
  \end{enumerate}
\end{remark}
Recall our convention that ``open map'' means ``continuous open map.''
\begin{definition}
  Let $X$ be a Hausdorff interpretable topological space.
  \begin{enumerate}
  \item $X$ is \emph{definably dominated} if there is an interpretable
    surjective open map $Y \to X$, for some definable set
    $Y \subseteq \Mm^n$ with the subspace topology.
  \item $X$ is \emph{manifold dominated} if there is an interpretable
    surjective open map $Y \to X$, for some definable
    manifold $Y$.
  \end{enumerate}
\end{definition}
\begin{lemma}\label{dom-trans}
  Let $X \to Y$ be an interpretable surjective  open map
  between two Hausdorff interpretable topological spaces $X$ and $Y$.
  \begin{enumerate}
  \item \label{dt1} If $X$ is definably dominated, then $Y$ is definably dominated.
  \item \label{dt2} If $X$ is manifold dominanted, then $Y$ is manifold dominated.
  \end{enumerate}
\end{lemma}
\begin{proof}
  We prove (\ref{dt1}); (\ref{dt2}) is similar.  Take a definable set $W \subseteq
  \Mm^n$ and an interpretable surjective  open map $W \to X$
  witnessing the fact that $X$ is definably dominated.  Then the
  composition $W \to X \to Y$ witnesses that $Y$ is definably
  dominated.
\end{proof}
\begin{lemma}
  If $X$ is manifold dominated, then $X$ is definably dominated.
\end{lemma}
\begin{proof}
  By Lemma~\ref{dom-trans}, it suffices to show that definable
  manifolds are definably dominated.  Let $X$ be a definable manifold.
  By definition there are definable open sets $U_i \subseteq
  \Mm^{n_i}$ for $1 \le i \le m$ and open embeddings $f_i : U_i \to X$
  which are jointly surjective.  The disjoint union $U_1 \sqcup \cdots
  \sqcup U_m$ is homeomorphic to a definable set $D \subseteq \Mm^N$
  for sufficiently large $N$.  The natural map $U_1 \sqcup \cdots
  \sqcup U_m \to X$ is an interpretable surjective open map.
  Therefore $X$ is definably dominated.
\end{proof}
\begin{definition} \label{adm-def}
  Let $X$ be a Hausdorff interpretable topological space.
  \begin{enumerate}
  \item $X$ is \emph{admissible} if $X$ is definably dominated and
    locally definable.
  \item $X$ is \emph{strongly admissible} if $X$ is manifold dominated
    and locally Euclidean.
  \end{enumerate}
  We say that an interpretable topology $\tau$ on an interpretable set
  $X$ is \emph{admissible} or \emph{strongly admissible} if the space
  $(X,\tau)$ is admissible or strongly admissible, respectively.
\end{definition}
In \cite[Section~4]{wj-o-minimal}, ``admissibility'' was used to refer
to the weaker condition of definable domination, without local
definability.  However, the frontier dimension inequality fails
without local definability (Proposition~\ref{frontier-dimension},
Remark~\ref{why-definably-dominated}).
\begin{example} \label{upsilon}
  If $X$ is a definable manifold, then $X$ is strongly admissible.
  (The identity map $\id_X$ witnesses that $X$ is manifold dominated.)
  Similarly, if $X \subseteq \Mm^n$ is a definable set with the
  subspace topology, then $X$ is admissible.
\end{example}
\begin{example} \label{gamma}
  The discrete topology on the value group $\Gamma$ is strongly
  admissible.  The valuation map $\Mm^\times \to \Gamma$ witnesses
  manifold-domination.  In contrast, the discrete topology on $\Mm$ is
  \emph{not} admissible, by Proposition~\ref{local-dim} below.
\end{example}
\begin{example} \label{half-compactification}
  Let $\Gamma_\infty$ be the extended value group $\Gamma \cup
  \{+\infty\}$.  Consider the topology on $\Gamma_\infty$ with basis
  \begin{equation*}
    \{\{\gamma\} : \gamma \in \Gamma\} \cup \{[\gamma,+\infty] : \gamma \in \Gamma\}.
  \end{equation*}
  The valuation map $\Mm \to \Gamma_\infty$ shows that $\Gamma_\infty$
  is manifold dominated.  However, $\Gamma_\infty$ is not locally
  definable.  The point $+\infty \in \Gamma_\infty$ has the property
  that every neighborhood is infinite and 0-dimensional.  This cannot
  happen in a definable set, where being 0-dimensional is equivalent
  to being finite.
\end{example}
\begin{example} \label{impl-1}
  Let $D \subseteq \Mm^2$ be the definable set $\{(x,y) \in \Mm^2 : x
  = 0 ~ \vee ~ y \ne 0\}$ with the subspace topology.  Then $D$ is
  admissible.  We will see below that $D$ is not manifold dominated
  (Example~\ref{impl-2}).
\end{example}

\subsection{Closure properties} \label{closure-props}
\begin{remark} \label{finite-strong}
  Let $X$ be a finite interpretable set with the discrete topology.
  Then $X$ is strongly admissible.  In fact, $X$ is a definable
  manifold.
\end{remark}
In the category of topological spaces, the class of  open
maps is closed under composition and base change.  Consequently, if
$f_i : Y_i \to X_i$ is an  open map for $i = 1, 2$, then the
product map $Y_1 \times Y_2 \to X_1 \times X_2$ is also an
open map.
\begin{proposition} \label{products-etc}
  The following classes of interpretable topological spaces are closed
  under finite disjoint unions and finite products.
  \begin{enumerate}
  \item \label{pe2} The class of locally definable spaces.
  \item \label{pe1} The class of locally Euclidean spaces.
  \item \label{pe3} The class of definably dominated spaces.
  \item \label{pe4} The class of manifold dominated spaces.
  \item \label{pe5} The class of admissible spaces.
  \item \label{pe6} The class of strongly admissible spaces.
  \end{enumerate}
\end{proposition}
\begin{proof}
  We focus on binary disjoint unions and binary products.  (The 0-ary
  cases are handled by Remark~\ref{finite-strong}.)
  
  Cases (\ref{pe1}) and (\ref{pe2}) are straightforward.  For
  (\ref{pe3}), let $X_1, X_2$ be two definably dominated spaces.  Let
  $f_i : Y_i \to X_i$ be a map witnessing definable domination for $i
  = 1, 2$.  Up to definable homeomorphism, $Y_1 \sqcup Y_2$ is a
  definable set, and so $Y_1 \sqcup Y_2 \to X_1 \sqcup X_2$ witnesses
  that $X_1 \sqcup X_2$ is definably dominated.  Similarly, $Y_1
  \times Y_2 \to X_1 \times X_2$ witnesses that $X_1 \times X_2$ is
  definably dominated.  The proof for (\ref{pe4}) is similar.  Then
  (\ref{pe2}) and (\ref{pe3}) give (\ref{pe5}), while (\ref{pe1}) and
  (\ref{pe4}) give (\ref{pe6}).
\end{proof}
Say that a class $\mathcal{C}$ of interpretable topological spaces is
``closed under subspaces'' if $\mathcal{C}$ contains every
interpretable subspace of every $X \in \mathcal{C}$.
\begin{proposition}\label{subspaces}
  The following classes of interpretable topological spaces are closed
  under subspaces:
  \begin{enumerate}
  \item \label{ss1} The class of locally definable spaces.
  \item \label{ss2} The class of definably dominated spaces.
  \item \label{ss3} The class of admissible spaces.
  \end{enumerate}
\end{proposition}
\begin{proof}
  (\ref{ss1}) is straightforward, and (\ref{ss1}) and (\ref{ss2}) give (\ref{ss3}), so it suffices to
  prove (\ref{ss2}).  Suppose $X$ is definably dominated, witnessed by an
  interpretable surjective open map $Y \to X$.  Let $X'$ be
  an interpretable subspace of $X$.  Let $Y'$ be the pullback $X'
  \times_X Y$:
  \begin{equation*}
    \xymatrix{ Y' \ar[d] \ar[r] & Y \ar[d] \\ X' \ar[r] & X.}
  \end{equation*}
  Then $Y'$ is an interpretable subspace of the definable set $Y
  \subseteq \Mm^n$, so $Y'$ is also a definable subset of $\Mm^n$.
  The map $Y' \to X'$ is open and surjective because it is the
  pullback of the open and surjective map $X' \to X$.  Then $Y' \to
  X'$ witnesses that $X'$ is definably dominated.
\end{proof}
The class of strongly admissible spaces is \emph{not} closed under
subspaces; local Euclideanity is not preserved.
\begin{remark} \label{open-obvious}
  If $X$ is strongly admissible and $U$ is an interpretable open
  subset, then $U$ (with the subspace topology) is strongly
  admissible.  Indeed, $U$ is Hausdorff and locally Euclidean, and if
  $Y$ is a definable manifold with an interpretable surjective
   open map $f : Y \to X$, then $f^{-1}(U)$ is a definable
  manifold and $f^{-1}(U) \to U$ shows that $U$ is manifold dominated.
\end{remark}

\subsection{Construction of strongly admissible topologies} \label{ss-construct}
In this section, we show that every interpretable set admits at least
one strongly admissible topology.  The proofs closely follow the
proofs in the o-minimal case \cite[Section~3]{wj-o-minimal}, and we
omit the proofs when no changes are needed.
\begin{definition} \label{def-oq}
  Let $X$ be a topological space and $E$ be an equivalence relation on
  $X$.  Consider $X/E$ with the quotient topology.
  \begin{enumerate}
  \item $E$ is an \emph{OQ equivalence
  relation} if $X \to X/E$ is an open map.
  \item If $D \subseteq X$, then the \emph{$E$-saturation of $D$},
    written $D^E$, is the union of the $E$-equivalence classes
    intersecting $D$.
  \end{enumerate}
\end{definition}
``OQ'' stands for ``open quotient.''  In \cite{wj-o-minimal}, OQ equivalence relations were called ``open equivalence relations,'' and the $E$-saturation was called the
``$E$-closure.''
\begin{remark} \label{criterion-for-oer}
  Let $X$ be a topological space with basis $\mathcal{B}$, and let $E$
  be an equivalence relation on $X$.  The following are equivalent:
  \begin{enumerate}
  \item $E$ is an OQ equivalence relation on $X$.
  \item For any open set $U \subseteq X$, the $E$-saturation $U^E$ is
    open.
  \item For any basic open set $B \in \mathcal{B}$, the $E$-saturation
    $B^E$ is open.
  \item \label{cfo4} If $p,p' \in X$ and $pEp'$ and $B \in \mathcal{B}$ is a basic
    neighborhood of $p$, then there is $B' \in \mathcal{B}$ a basic
    neighborhood of $p'$, such that $\forall x \in B' ~ \exists y \in
    B ~ xEy$.  (That is, $B' \subseteq B^E$.)
  \end{enumerate}
\end{remark}
\begin{fact}[{\cite[Lemma~3.12]{wj-o-minimal}}]
  \label{why-open}
  Let $X$ be an interpretable topological space and $E$ be an
  interpretable OQ equivalence relation on $X$.  Then the quotient
  topology on $X/E$ is interpretable.
\end{fact}
When $E$ is not an OQ equivalence relation, the quotient topology on
$X/E$ can fail to be interpretable.  For example, this happens when $X
= \Mm^2$ and $E$ is the equivalence relation collapsing the $x$-axis
to a point.
\begin{fact}[{\cite[Lemma~3.13]{wj-o-minimal}}]
  \label{open-to-subs}
  Let $X$ be a topological space and $E$ be an OQ equivalence
  relation on $X$.  Let $X'$ be an open subset of $X$ and let $E'$ be
  the restriction $E \restriction X$.
  \begin{enumerate}
  \item \label{ots1} $E'$ is an OQ equivalence relation on $X$.
  \item $X'/E' \to X/E$ is an open embedding.
  \end{enumerate}
\end{fact}
\begin{definition}
  Let $X$ be a non-empty interpretable set.  An interpretable subset
  $Y \subseteq X$ is a \emph{large subset} if $\dim(Y \setminus X) <
  \dim(X)$.
\end{definition}
This terminology follows \cite[Definition~1.11]{Pillay-G-in-o} and
\cite[Definition~8.3]{interpretable-groups}.  In \cite{wj-o-minimal},
large sets were called ``full sets.''
\begin{remark} \phantomsection \label{large-remark}
  \begin{enumerate}
  \item If $X$ is a large subset of $Y$, then $\dim(X) = \dim(Y)$.
  \item \label{lr2} If $X$ is a large subset of $Y$, and $Y$ is a large subset of
    $Z$, then $X$ is a large subset of $Z$.
  \item If $\dim(X) = 0$, the only large subset is $X$ itself.
  \end{enumerate}
\end{remark}
\begin{remark} \label{definable-to-manifold}
  If $X \subseteq \Mm^n$ is definable, then there is a large open
  subset $X' \subseteq X$ such that $X'$ is a definable manifold.  If
  $X$ is $M$-definable for a small model $M \preceq \Mm$, then we can
  take $X'$ to be $M$-definable.  If $k = \dim(X)$, then we can take
  $X'$ to be a $k$-dimensional manifold in the strong sense that $X'$
  is covered by finitely many open sets definably homeomorphic to open
  subsets of $\Mm^k$.

  This is obvious from cell decomposition results and dimension
  inequalities, but we give a proof for completeness.
\end{remark}
\begin{proof}
  Say that a non-empty set $C \subseteq \Mm^n$ is a \emph{t-cell}
  \cite[Definition~2.6]{p-minimal-cells} if there is a coordinate
  projection $\pi : \Mm^n \to \Mm^k$ such that $\pi(C)$ is open in
  $\Mm^k$, and $C \to \pi(C)$ is a homeomorphism.  By \cite[\S4,
    Theorem~$1.1'$]{vdDS}, we can write $X$ as a finite disjoint union
  of t-cells $\bigcup_{i = 1}^m C_i$.  If $X$ is $M$-definable, we can
  take the $C_i$ to be $M$-definable.  Without loss of generality,
  $C_1,\ldots,C_j$ have dimension $k$ and $C_{j+1},\ldots,C_m$ have
  dimension $< k$.  Let $D_i$ be the union of the cells other than
  $C_i$.  For $i \le j$, let
  \begin{equation*}
    U_i = X \setminus \overline{D_i} = C_i \setminus \overline{D_i} =
    C_i \setminus \partial D_i
  \end{equation*}
  and take $X' = \bigcup_{i = 1}^j U_i$.  Each set $U_i$ is an open
  subset of $X$, so $X'$ is an open subset.  For each $i \le j$, $U_i$
  is an open subset of $C_i$, which is definably homeomorphic to an
  open subset of $\Mm^k$, and therefore $U_i$ is also definably
  homeomorphic to an open subset of $\Mm^k$.  Therefore $X'$ is a
  $k$-dimensional manifold in the strong sense.  To see that $X'$ is a
  large subset of $X$, note that
  \begin{equation*}
    X \setminus X' = \bigcup_{i = 1}^j (C_i \cap \partial D_i) \cup \bigcup_{i = {j+1}}^m C_i.
  \end{equation*}
  For $i > j$, $\dim(C_i) < k$ by assumption.  The frontier dimension
  inequality \cite[Theorem~3.5]{p-minimal-cells} shows that $\dim(C_i
  \cap \partial D_i) \le \dim(\partial D_i) < \dim(D_i) \le k$.
  Therefore $\dim(X \setminus X') < k = \dim(X')$.
\end{proof}
\begin{definition}
  Let $D$ be an $A$-interpretable set.  An element $b \in D$ is
  \emph{generic in $D$ (over $A$)} if $\dim(b/A) = \dim(D)$.
\end{definition}
\begin{warning}
  If $D$ is 0-dimensional, then every element of $D$ is generic.
\end{warning}
\begin{lemma} \label{generic-eucl}
  If $X \subseteq \Mm^n$ is non-empty and $A$-definable, and $p$ is
  generic in $X$ (over $A$), then $X$ is locally Euclidean on an open
  neighborhood of $p$.
\end{lemma}
\begin{proof}
  By Remark~\ref{definable-to-manifold} there is a large open subset $X'
  \subseteq X$ that is a definable manifold.  Let $B$ be a set of
  parameters defining $X'$.  Moving $B$ and $X'$ by an automorphism
  over $A$, we may assume $p \dimind_A B$.  Then $\dim(p/B) =
  \dim(p/A) = \dim(X) > \dim(X \setminus X')$, and so $p$ is not in
  the $B$-definable set $X \setminus X'$.  Then $p \in X'$, and $X'$
  is the desired open neighborhood of $p$.
\end{proof}

The following lemma, which parallels \cite[Lemma~3.15]{wj-o-minimal},
collects some useful tools.
\begin{lemma} \label{tricks}
  Let $X \subseteq \Mm^n$ be $A$-definable, endowed with the subspace
  topology.
  \begin{enumerate}
  \item \label{tr1} Let $D$ be a definable subset of $X$.  Let $\partial D$ be the
    frontier of $D$, within the topological space $X$.  Then
    $\dim(\partial D) < \dim(D)$.
  \item \label{tr2} Suppose $P \subseteq X$ is definable or $\vee$-definable over
    $A$, and suppose $P$ contains every element that is generic in
    $X$.  Then there is an $A$-definable large open subset $X'
    \subseteq X$ such that $X' \subseteq P$.
  \item \label{tr3} Suppose $b \in X$ and $U$ is a neighborhood of $b$ in the
    topological space $X$.  Let $S \subseteq \Mm^\eq$ be a small set
    of parameters.  Then there is a smaller open neighborhood $b \in U'
    \subseteq U$ such that $\ulcorner U' \urcorner \dimind_A bS$.
  \end{enumerate}
\end{lemma}
\begin{proof}
  Part (\ref{tr1}) is \cite[Theorem~3.5]{p-minimal-cells}.  Part
  (\ref{tr2}) is proved analogously to
  \cite[Lemma~3.15(2)]{wj-o-minimal}, making use of
  Fact~\ref{continuity-01}(\ref{dos}).  Part (\ref{tr3}) requires a
  new argument.  Take $U' = X \cap B$, where $B$ is a sufficiently
  small ball in $\Mm^n$ around $b$.  A simple calculation shows that
  the set of balls in $\Mm^n$ is 0-dimensional.
  Therefore\[\dim(\ulcorner B \urcorner/\varnothing) = \dim(\ulcorner
  B \urcorner / A) = \dim(\ulcorner B \urcorner / AbS) = 0,\] implying
  $\ulcorner B \urcorner \dimind_A bS$.  But $\ulcorner U' \urcorner
  \in \dcl^\eq(A \ulcorner B \urcorner)$, and so $\ulcorner U'
  \urcorner \dimind_A bS$.
\end{proof}
The next three lemmas show that given a definable set $X$ and
definable equivalence relation $E$, we can improve the nature of $X/E$
by repeatedly replacing $X$ with a large open definable subset.
\begin{lemma} \label{step-1}
  Let $X \subseteq \Mm^n$ be $A$-definable and non-empty and let $E$
  be an $A$-definable equivalence relation on $X$.  Then there is an
  $A$-definable large open subset $X' \subseteq X$ such that the
  restriction $E \restriction X'$ is an OQ equivalence relation on
  $X'$.
\end{lemma}
\begin{proof}
  The proof of \cite[Proposition~3.16]{wj-o-minimal} works verbatim,
  replacing $\ind^{\mathrm{\th}}$ with $\dimind$.
\end{proof}
\begin{lemma} \label{step-2}
  Let $X \subseteq \Mm^n$ be $A$-definable and non-empty and let $E$
  be an $A$-definable OQ equivalence relation on $X$. Then there is an
  $A$-definable large open subset $X' \subseteq X$ such that $X'/E := X'/(E
  \restriction X')$ is Hausdorff.
\end{lemma}
\begin{proof}
  The proof of \cite[Proposition~3.18]{wj-o-minimal} works verbtaim,
  replacing $\ind^{\mathrm{\th}}$ with $\dimind$.
\end{proof}
\begin{lemma} \label{step-3}
  Let $X \subseteq \Mm^n$ be $A$-definable and non-empty and let $E$
  be an $A$-definable OQ equivalence relation on $X$ such that $X/E$
  is Hausdorff.  Then there is an $A$-definable large open subset $X'
  \subseteq X$ such that $X'/E$ is locally Euclidean.
\end{lemma}
\begin{proof}
  The proof of \cite[Proposition~3.20]{wj-o-minimal} works verbtaim,
  replacing $\ind^{\mathrm{\th}}$ with $\dimind$.  Definable
  compactness continues to work properly thanks to
  Fact~\ref{dc-fact-pcf}.  The use of the $\dcl(-)$ pregeometry on the
  home sort does not present any problems because of
  Fact~\ref{gagelman-exchange} and the fact that $\acl^\eq(A) \cap \Mm
  = \dcl^\eq(A) \cap \Mm$ for $A \subseteq \Mm^\eq$ in $\pCF$.
\end{proof}
\begin{theorem} \label{thm-pi}
  Let $X \subseteq \Mm^n$ be a non-empty definable set, and $E$ be a
  definable equivalence relation on $X$.  There is a large open subset
  $X' \subseteq X$ such that
  \begin{itemize}
  \item $X'$ is a definable manifold
  \item $E \restriction X'$ is an OQ equivalence relation on $X'$
  \item $X'/E$ is Hausdorff
  \item $X'/E$ is locally Euclidean.
  \end{itemize}
  If $X$ and $E$ are $M$-definable for some small model $M \preceq
  \Mm$, then we can take $X'$ to be $M$-definable.
\end{theorem}
\begin{proof}
  The proof is similar to \cite[Theorem~3.14]{wj-o-minimal}.  Take a
  small model $M$ defining $X$ and $E$.  First apply
  Remark~\ref{definable-to-manifold} to obtain an $M$-definable large
  open subset $X_1 \subseteq X$ such that $X_1$ is a definable
  manifold.  Let $E_1 = E \restriction X_1$.  Then apply
  Lemma~\ref{step-1} to get an $M$-definable large open subset $X_2
  \subseteq X_1$ such that if $E_2 = E \restriction X_2$, then $E_2$
  is an OQ equivalence relation on $X_2$.  Then apply
  Lemma~\ref{step-2} to get an $M$-definable large open subset $X_3
  \subseteq X_2$ such that if $E_3 = E \restriction X_3$, then
  $X_3/E_3$ is Hausdorff.  (The relation $E_3$ is an OQ equivalence
  relation on $X_3$ by Fact~\ref{open-to-subs}(\ref{ots1}).)  Then apply
  Lemma~\ref{step-3} to get an $M$-definable large open subset $X_4
  \subseteq X_3$ such that if $E_4 = E \restriction X_4$, then
  $X_4/E_4$ is locally Euclidean.  By Fact~\ref{open-to-subs}, $E_4$
  is an OQ equivalence relation on $X_4$, and the quotient space
  $X_4/E_4$ is an open subspace of $X_3/E_3$.  In particular,
  $X_4/E_4$ is Hausdorff.  By Remark~\ref{large-remark}(\ref{lr2}), $X_4$ is a
  large open subset of $X$.  Thus $X_4$ is a definable manifold.  Take
  $X' = X_4$.
\end{proof}
\begin{theorem} \label{construction-1}
  Let $Y$ be an interpretable set.  Then $Y$ admits at least one
  strongly admissible topology.  If $Y$ is $M$-interpretable for a
  small model $M \preceq \Mm$, then $Y$ admits an $M$-interpretable
  strongly admissible topology.
\end{theorem}
\begin{proof}
  The proof is similar to \cite[Proposition~4.2]{wj-o-minimal}.  Write
  $Y$ as $X/E$ for some $M$-definable set $X \subseteq \Mm^n$ and
  $M$-definable equivalence relation $E$.  Proceed by induction on
  $\dim(X)$.  If $\dim(X) = -\infty$, then $X$ and $Y$ are empty, and
  the unique topology on $Y$ is strongly admissible.  Suppose $\dim(X)
  \ge 0$, so $X$ and $Y$ are non-empty.  By Theorem~\ref{thm-pi},
  there is an $M$-definable large open subset $X' \subseteq X$ such
  that $X'$ is a definable manifold, and if $E' = E \restriction X'$,
  then $E'$ is an OQ equivalence relation on $X'$, and the quotient
  $X'/E'$ is Hausdorff and locally Euclidean.  Then the quotient
  topology on $X'/E'$ is interpretable by Fact~\ref{why-open}, and the
  map $X' \to X'/E'$ is a  surjective interpretable open
  map, because $E'$ is an OQ equivalence relation.  Therefore $X'
  \to X'/E'$ witnesses that $X'/E'$ (with the quotient topology) is
  strongly admissible.

  Let $Y' = X'/E'$, and let $Y'' = Y \setminus Y''$.  Then $Y$ is a
  disjoint union of $Y'$ and $Y''$.  Let $X''$ be the preimage of
  $Y''$ in $X$, and let $E'' = E \restriction X''$.  Then $Y'' =
  X''/E''$.  Moreover, $X'' \cap X' = \varnothing$, so $X'' \subseteq
  X \setminus X'$ and then $\dim(X'') < \dim(X)$ as $X'$ is large.  By
  induction, $Y''$ admits a strongly admissible topology.  So $Y'$ and
  $Y''$ both admit strongly admissible topologies.  The disjoint union
  topology on $Y$ is strongly admissible by
  Proposition~\ref{products-etc}.
\end{proof}

\subsection{Tameness in admissible topologies} \label{ss-tame}
In this section, we verify that the topological tameness theorems for definable sets
extend to admissible topological spaces.
\begin{lemma} \label{side-shrink}
  Let $X$ be a definably dominated interpretable topological space,
  interpretable over a small set of parameters $A$.  Let $S$ be a
  small set of parameters.  Let $b \in X$ be a point.  Let $U \ni b$
  be a neighborhood of $b$ in $X$.  Then there is a smaller
  neighborhood $b \in U' \subseteq U$ such that $\ulcorner U'
  \urcorner \dimind_A bS$.
\end{lemma}
This is similar to but slightly stronger than
\cite[Lemma~4.5]{wj-o-minimal}, and nearly the same proof works:
\begin{proof}
  Let $f : Y \to X$ be a map witnessing definable domination.  In
  particular, $Y$ is a definable set, and $f$ is a 
  surjective open map.  Let $A' \supseteq A$ be a small set over which
  $X, Y$, and $f$ are interpretable.  Moving $f, Y, A'$ by an
  automorphism over $A$, we may assume $A' \dimind_A pS$ by
  Proposition~\ref{extension-2}.  Take $\tilde{p} \in Y$ with
  $f(\tilde{p}) = p$.  Then $f^{-1}(U)$ is a neighborhood of
  $\tilde{p}$.  By Lemma~\ref{tricks}(\ref{tr3}), there is a smaller
  neighborhood $U''$ of $\tilde{p}$ in $Y$ such that $\ulcorner U''
  \urcorner \dimind_{A'} \tilde{p}S$.  Take $U' = f(U'')$; this is
  a neighborhood of $p$ because $f$ is an open map.  Then $\ulcorner
  U' \urcorner \dimind_{A'} pS$.  As $A' \dimind_A pS$, left
  transitivity gives $\ulcorner U' \urcorner \dimind_A pS$.
\end{proof}
\begin{definition}
  Let $X$ be an interpretable topological space and $p \in X$ be a
  point.  The \emph{local dimension} of $X$ at $p$, written
  $\dim_p(X)$, is the minimum of $\dim(U)$ as $U$ ranges over
  neighborhoods of $p$ in $X$.  More generally, if $D$ is an
  interpretable subset of $X$, then $\dim_p(D)$ is the minimum of
  $\dim(U \cap D)$ as $U$ ranges over neighborhoods of $p$ in $X$.
\end{definition}
\begin{proposition} \label{local-dim}
  Let $X$ be a definably dominated interpretable topological space.
  Then $\dim(X) = \max_{p \in X} \dim_p(X)$.
\end{proposition}
\begin{proof}
  The proof of \cite[Proposition~4.6]{wj-o-minimal} goes through
  verbatim, replacing $\ind^{\mathrm{\th}}$ with $\dimind$,
  using Lemma~\ref{side-shrink}.
\end{proof}
\begin{corollary} \label{container}
  Let $X$ be an interpretable set with $n = \dim(X)$.  Then there is a
  non-empty open definable set $D \subseteq \Mm^n$ and an
  interpretable injection $D \hookrightarrow X$.
\end{corollary}
\begin{proof}
  By Theorem~\ref{construction-1}, there is a strongly admissible
  topology on $X$.  By Proposition~\ref{local-dim} there is a point $p
  \in X$ with $\dim_p(X) = n$.  By local Euclideanity, $p$ has a
  neighborhood that is interpretably homeomorphic to a ball in $\Mm^m$
  for some $m$.  Clearly $m = \dim_p(X)$.
\end{proof}
\begin{proposition} \label{frontier-dimension}
  Let $X$ be an admissible interpretable topological space and $D$ be
  an interpretable subset.  Then $\dim \partial D < \dim D$, $\dim
  \bd(D) < \dim(X)$, and $\dim \overline{D} = \dim D$.
\end{proposition}
\begin{proof}
  The proof of \cite[Proposition~4.7]{wj-o-minimal} goes through
  almost verbatim, using Proposition~\ref{local-dim}.
\end{proof}
\begin{remark} \label{why-definably-dominated}
  Proposition~\ref{frontier-dimension} fails if we replace
  ``admissible'' with the weaker condition ``definably dominated.''
  If $\Gamma_\infty$ is as in Example~\ref{half-compactification},
  then the subset $\Gamma \subseteq \Gamma_\infty$ has the same
  dimension (zero) as its frontier $\{+\infty\}$.
\end{remark}

\begin{proposition} \label{mostly-euclidean}
  Let $X$ be an admissible interpretable topological space.  Then
  there is a large open subset $X' \subseteq X$ such that $X'$ (with
  the subspace topology) is locally Euclidean.
\end{proposition}
\begin{proof}
  The proof is similar to \cite[Proposition~4.7(3) and Lemma~4.9]{wj-o-minimal}.
  Let $A$ be a small set of parameters over which $X$ is
  interpretable.  Let $X_{\mathrm{Eu}}$ be the locally Euclidean locus
  of $X$, the set of points $p \in X$ such that some neighborhood of
  $p$ is interpretably homeomorphic to an open definable subset of
  $\Mm^n$ for some $n$.  It is easy to see that $X_{\mathrm{Eu}}$ is
  $\vee$-definable, a small union of $A$-interpretable subsets of $X$.
  
  Let $Y$ be the union of the $A$-interpretable subsets $D \subseteq
  X$ with $\dim(D) < \dim(Y)$.  Like $X_{\mathrm{Eu}}$, the set $Y$ is
  $\vee$-definable over $A$.
  \begin{description}
  \item[Case 1:] $X_{\mathrm{Eu}} \cup Y \subsetneq X$.  Take $b \in X
    \setminus (X_{\mathrm{Eu}} \cup Y)$.  Then $\dim(b/A) = \dim(X)$
    by Proposition~\ref{continuity-01}(\ref{dos}) and the fact that $b
    \notin Y$.  By local definability and Lemma~\ref{side-shrink},
    there is a neighborhood $U$ of $b$ such that $U$ is homeomorphic
    to a definable subset of $\Mm^n$ and $\ulcorner U \urcorner
    \dimind_A b$.  Let $f : U \to \Mm^n$ be an interpretable
    topological embedding.  Let $B \supseteq A$ be a small set of
    parameters defining $U$ and $f$.  Moving $B, f$ by an automorphism
    in $\Aut(\Mm/A \ulcorner U \urcorner)$, we may assume $B
    \dimind_{A \ulcorner U \urcorner} b$.  By left transitivity, $B
    \dimind_A b$.  Then
    \begin{equation*}
      \dim(X) = \dim(b/A) = \dim(b/B) = \dim(f(b)/B) \le \dim(f(U)) =
      \dim(U) \le \dim(X).
    \end{equation*}
    Therefore $f(b)$ is generic in the $B$-definable set $f(U)$.  By
    Lemma~\ref{generic-eucl}, $f(U)$ is locally Euclidean on an open
    neighborhood of $f(b)$.  Then $U$ is locally Euclidean on an open
    neighborhood of $b$, contradicting the fact that $b \notin
    X_{\mathrm{Eu}}$.
  \item[Case 2:] $X = X_{\mathrm{Eu}} \cup Y$.  Then saturation gives
    $X = X_{\mathrm{Eu}} \cup Y_0$ where $Y_0$ is a finite union of
    $A$-interpretable subsets of $X$ of lower dimension.  In
    particular, $\dim(Y_0) < \dim(X)$.  By
    Proposition~\ref{frontier-dimension}, the closure $\overline{Y_0}$
    has lower dimension than $X$.  Take $X' = X \setminus
    \overline{Y_0}$. \qedhere
  \end{description}
\end{proof}
\begin{corollary} \label{generically-n-manifold}
  If $X$ is an admissible interpretable topological space of dimension
  $n$, then there is a large open subset $X' \subseteq X$ such that
  $X'$ is everywhere locally homeomorphic to $\Mm^n$.
\end{corollary}
\begin{proof}
  By Proposition~\ref{mostly-euclidean}, we may pass to a large open
  subset and assume $X$ is locally Euclidean.  Then partition $X$ as
  $\bigcup_{i = 0}^n X_i$, where $X_i = \{p \in X : \dim_p(X) = i\}$.
  Local Euclideanity ensures that $p \mapsto \dim_p(X)$ is locally
  constant, so each $X_i$ is open.  By Proposition~\ref{local-dim},
  $\dim(X_i) = i$ for non-empty $X_i$.  Therefore $X' := X_n$ is non-empty,
  and its complement has dimension $< n$.
\end{proof}
A slightly different argument from Proposition~\ref{mostly-euclidean} gives
the following:
\begin{lemma} \label{some-open}
  Let $X$ be a non-empty definably dominated interpretable topological
  space.  Then $X$ has a non-empty open subset $X' \subseteq X$ that is
  strongly admissible.
\end{lemma}
\begin{proof}
  Take $Y \subseteq \Mm^n$ definable and $f : Y \to X$ an
  interpretable  surjective open map.  Let $E$ be the kernel
  equivalence relation $\{(x,y) \in Y^2 : f(x) = f(y)\}$.  Then we can
  identify $X$ with $Y/E$ on the level of sets.  Moreover, $X$ is
  homeomorphic to $Y/E$ with the quotient topology:
  \begin{itemize}
  \item If $U \subseteq X$ is open, then $f^{-1}(U)$ is open because
    $f$ is continuous.
  \item If $f^{-1}(U)$ is open then $U = f(f^{-1}(U))$ is open because
    $f$ is a surjective open map.
  \end{itemize}
  Thus we may identify $X$ with the quotient topology $Y/E$.  The fact
  that $Y \to Y/E$ is open means that $E$ is an OQ equivalence
  relation.

  By Theorem~\ref{thm-pi}, there is a large (hence non-empty) open
  subset $Y' \subseteq Y$ such that $Y'$ is a definable manifold, and
  $Y'/E$ is locally Euclidean.  By Fact~\ref{open-to-subs}, $Y'/E$ is
  an open subspace of $Y/E = X$, and $Y' \to Y'/E$ is an open map
  (i.e., $E \restriction Y'$ is an OQ equivalence relation).  Then $X'
  := Y'/E$ is strongly admissible.
\end{proof}
\begin{proposition} \label{gen-con}
  Let $X, Y$ be admissible topological spaces.  Let $f : X \to Y$ be
  interpretable.
  \begin{enumerate}
  \item There is a large open subset $X' \subseteq X$ on which $f$ is
    continuous.
  \item We can write $X$ as a finite disjoint union of locally closed
    interpretable subsets on which $f$ is continuous.
  \end{enumerate}
\end{proposition}
\begin{proof}
  The proof of \cite[Proposition~4.12]{wj-o-minimal} works,
  essentially verbatim, using generic continuity of definable
  functions in $\pCF$.
\end{proof}
\begin{corollary} \label{gen-con-2}
  Let $X, Y$ be $A$-interpretable admissible topological spaces, and
  $f : X \to Y$ be $A$-interpretable.  If $b \in X$ is generic over
  $A$, then $f$ is continuous at $p$.
\end{corollary}

\section{Admissible groups} \label{adm-sec-2}

\subsection{Uniqueness}
\begin{lemma} \label{to-admissible}
  Let $G$ be an interpretable group, and let $\tau$ be an
  interpretable topology that is invariant under left translations.
  If $\tau$ is definably dominated, then $\tau$ is admissible and
  locally Euclidean.
\end{lemma}
\begin{proof}
  Lemma~\ref{some-open} gives a non-empty open subset of $(G,\tau)$
  that is locally Euclidean.  By translation invariance, $(G,\tau)$ is
  locally Euclidean everywhere.  Then $(G,\tau)$ is definably
  dominated and locally definable, i.e., admissible.
\end{proof}
\begin{lemma} \label{uniqueness-0}
  Let $G$ be an interpretable group.  Then $G$ admits at most one
  admissible topology $\tau$ that is invariant under
  left-translations.
\end{lemma}
\begin{proof}
  Suppose $\tau_1$ and $\tau_2$ are left-invariant admissible
  topologies on $G$.  By Proposition~\ref{gen-con}, the map $\id_G :
  (G,\tau_1) \to (G,\tau_2)$ is continuous at at least one point $a
  \in G$.  If $\delta \in G$ and $f(x) = \delta \cdot x$, then the
  composition
  \begin{equation*}
    (G,\tau_1) \stackrel{f}{\to} (G,\tau_1) \stackrel{\id_G}{\to}
    (G,\tau_2) \stackrel{f^{-1}}{\to} (G,\tau_2)
  \end{equation*}
  is continuous at $\delta^{-1} a$.  That is, $\id_G : (G,\tau_1) \to
  (G,\tau_2)$ is continuous at $\delta^{-1} a$, for arbitrary
  $\delta$.  Then $\id_G : (G,\tau_1) \to (G,\tau_2)$ is continuous
  (everywhere).  Similarly, $\id_G : (G,\tau_2) \to (G,\tau_1)$ is
  continuous, implying $\tau_1 = \tau_2$.
\end{proof}
Lemma~\ref{uniqueness-0} shows that an interpretable group admits at
most one admissible group topology.  In the following section, we will
see that every interpretable group atmits \emph{at least one}
admissible group topology.  For now, we draw a useful consequence
of the above lemmas.
\begin{lemma} \label{upgrade}
  Let $G$ be an interpretable group.  Let $\tau$ be a definably
  dominated interpretable topology on $G$ that is invariant under left
  translations.  Then $\tau$ is an admissible group topology on $G$.
\end{lemma}
\begin{proof}
  First, $\tau$ is admissible by Lemma~\ref{to-admissible}.  We claim
  $\tau$ is invariant under right translations.  Take $g \in G$, and
  let $\tau'$ be the image of $\tau$ under the right translation $x
  \mapsto g \cdot x$.  Then $\tau'$ is invariant under left
  translations (because left and right translations commute).  By
  Lemma~\ref{uniqueness-0}, $\tau' = \tau$.  Thus $\tau$ is
  right-invariant.  In particular, right translations are continuous.

  Let $\tau^{-1}$ be the image of $\tau$ under the inverse map.  The
  fact that $\tau$ is right-invariant implies that $\tau^{-1}$ is left
  invariant.  Then $\tau^{-1} = \tau$, implying that the inverse map
  is continuous.

  Finally we show that the group operation $m : G \times G \to G$ is
  continuous.  By Proposition~\ref{gen-con} (and
  Proposition~\ref{products-etc}), the group operation $m$ is
  continuous at at least one point $(a,b)$.  For any $\delta, \epsilon
  \in G$, we have
  \begin{equation*}
    m(x,y) = \epsilon^{-1} \cdot m(\epsilon \cdot x, y \cdot \delta) \cdot \delta^{-1}.
  \end{equation*}
  The right hand side is continuous at $(x,y) = (\epsilon^{-1} \cdot
  a, b \cdot \delta^{-1})$, and therefore so is the left hand side.
  As $\epsilon, \delta$ are arbitrary, $m$ is continuous everywhere.
\end{proof}

\subsection{Existence} \label{sec:ex}
In this section we show that each interpretable group admits at least
one strongly admissible group topology.
\begin{lemma} \label{genericky-tool}
  Let $G$ be an interpretable group and let $U \subseteq G$ be a large
  subset.
  \begin{enumerate}
  \item \label{gt1} Finitely many left-translates of $U$ cover $G$.
  \item \label{gt2} For any $x, y \in G$, there is a left-translate $g \cdot U$
    containing both $x$ and $y$.
  \end{enumerate}
\end{lemma}
\begin{proof}
  Take a small set of parameters $A$ defining $G$ and $U$.
  \begin{enumerate}
  \item 
    Take $(g_1,\ldots,g_{n+1})$ generic in $G^{n+1}$ over $A$.  Note
    that $\dim(g_i/A,g_1,\ldots,g_{i-1}) = \dim(g_i/A) = \dim(G)$ for
    each $i$, and so $g_i \dimind_A g_1,\ldots,g_{i-1}$.

    We claim $G \subseteq \bigcup_{i = 1}^{n+1} g_i \cdot U$.  Fix any
    $h \in U$.  For $0 \le i \le n+1$ let $d_i =
    \dim(h/A,g_1,\ldots,g_i)$.  Then
    \begin{equation*}
      0 \le d_{n+1} \le d_n \le \cdots \le d_1 \le d_0 = \dim(h/A)
      \le \dim(G) = n.
    \end{equation*}
    It is impossible that $0 \le d_{n+1} < d_n < \cdots < d_0 \le n$,
    so there is some $1 \le i \le n+1$ such that $d_{i-1} = d_i$,
    i.e., $\dim(h/A,g_1,\ldots,g_{i-1}) = \dim(h/A,g_1,\ldots,g_i)$.
    Then
    \begin{equation*}
      h \dimind_{A, g_1,\ldots,g_{i-1}} g_i \text{ and } g_1,\ldots,g_{i-1} \dimind_A g_i,
    \end{equation*}
    so $h \dimind_A g_i$ by left transitivity.  Then $\dim(g_i/Ah) =
    \dim(g_i/A) = \dim(G)$, and $g_i$ is generic in $G$ over $Ah$.
    Then $\dim(g_i^{-1} \cdot h / Ah) = \dim(g_i / Ah) = \dim(G)$, and
    so $g_i^{-1} \cdot h$ must be in the $Ah$-interpretable large subset
    $U$, implying $h \in g_i \cdot U$.
  \item Take $g \in G$ generic over $Axy$.  Then $\dim(g^{-1} \cdot x
    / Axy) = \dim(g/Axy) = \dim(G)$, and so $g^{-1} \cdot x$ is in the
    $Axy$-interpretable large subset $U$.  Then $x \in g \cdot U$, and
    similarly $y \in g \cdot U$. \qedhere
  \end{enumerate}
\end{proof}

\begin{lemma} \label{finite-cover}
  Let $X$ be a Hausdorff interpretable topological space.  Let $U_1,
  \ldots, U_n$ be finitely many interpretable open subsets covering
  $X$.  If each $U_i$ is strongly admissible, then $X$ is strongly
  admissible.
\end{lemma}
\begin{proof}
  The space $X$ is locally Euclidean because it is covered by locally
  Euclidean spaces.  The disjoint union $U_1 \sqcup \cdots \sqcup U_n$
  is manifold dominated by Proposition~\ref{products-etc}.  The map
  $U_1 \sqcup \cdots \sqcup U_n \to U_1 \cup \cdots \cup U_n = X$ is
  an interpretable surjective  open map, and so $X$ is
  manifold dominated by Lemma~\ref{dom-trans}.  Then $X$ is Hausdorff,
  locally Euclidean, and manifold dominated.
\end{proof}

\begin{proposition} \label{construction-2}
  Let $G$ be an interpretable group.  Then there is a strongly
  admissible group topology $\tau$ on $G$.
\end{proposition}
\begin{proof}
  By Theorem~\ref{construction-1} there is a strongly admissible
  topology $\tau_0$ on $G$.  Let $C$ be a small set of parameters over
  which everything is defined.  If $a, b \in G$, let $a \preceq b$
  mean that the map $x \mapsto b \cdot a^{-1} \cdot x$ is
  $\tau_0$-continuous at $a$.  Then $\preceq$ is a $C$-interpretable
  preorder on $G$.  Let $\approx$ be the associated $C$-interpretable
  equivalence relation.
  \begin{claim}
    If $(a,b) \in G^2$ is generic over $C$, then $a \approx b$.
  \end{claim}
  \begin{claimproof}
    By symmetry it suffices to show $a \preceq b$.  Let $\delta = b
    \cdot a^{-1}$.  Then $\dim(a,\delta/C) = \dim(a,b/C) = 2 \dim(G)$,
    which implies that $a$ is generic in $G$ over $C \delta$.  By
    Corollary~\ref{gen-con-2}, the map $x \mapsto \delta \cdot x$ is
    $\tau_0$-continuous at $a$.
  \end{claimproof}
  Fix $a_0 \in G$ generic over $C$.  If $b \in G$ is generic over
  $Ca_0$, then $b \approx a_0$.  Therefore the equivalence class of
  $a_0$ is a large subset of $G$.  Let $U$ be the $\tau_0$-interior of
  this equivalence class.  Then $U$ is a large subset of $G$ by
  Proposition~\ref{frontier-dimension}.  We have constructed a large
  $\tau_0$-open subset $U \subseteq G$ such that if $a, b \in U$, then
  the map $x \mapsto b \cdot a^{-1} \cdot x$ is $\tau_0$-continuous at
  $x = a$.

  Consider $G \times U$, where $G$ has the discrete topology(!), $U$
  has the topology $\tau_0 \restriction U$, and $G \times U$ has the
  product topology.
  \begin{claim}
    Let $E$ be the equivalence relation on $G \times U$ given by
    \begin{equation*}
      (x,y)E(x',y') \iff x \cdot y = x' \cdot y'.
    \end{equation*}
    Then $E$ is an OQ equivalence relation on $G \times U$
    (Definition~\ref{def-oq}).
  \end{claim}
  \begin{claimproof}
    Fix an interpretable basis $\mathcal{B}$ for $U$.  Then $\{\{a\}
    \times B : a \in G, ~ B \in \mathcal{B}\}$ is a basis for $G
    \times U$.
    
    We verify that $E$ satisfies criterion (\ref{cfo4}) of
    Remark~\ref{criterion-for-oer}.  Fix $p = (a,b)$ and $p' = (c,d)$
    in $G \times U$, with $a \cdot b = c \cdot d$.  Fix a basic
    neighborhood $\{a\} \times B$ of $(a,b)$.  We must find a basic
    neighborhood $\{c\} \times B'$ of $(c,d)$ such that every element
    of $\{c\} \times B'$ is equivalent to an element of $\{a\} \times
    B$.  Note that $b \cdot d^{-1} = a^{-1} \cdot c$.  As $b, d \in
    U$, the map
    \begin{equation*}
      f(x) = b \cdot d^{-1} \cdot x = a^{-1} \cdot c \cdot x
    \end{equation*}
    is continuous at $x = d$, and maps $d$ to $b$.  As $B$ is a
    neighborhood of $b$, there is a neighborhood $B'$ of $d$ such that
    $f(B') \subseteq B$.  If $(c,x) \in \{c\} \times B'$ then
    $(c,x)E(a,f(x))$ and $(a,f(x)) \in \{a\} \times B$.  So every
    element of $\{c\} \times B'$ is equivalent to an element of $\{a\}
    \times B$, as required.
  \end{claimproof}
  Let $\tau$ be the quotient equivalence relation on $(G \times U)/E =
  G$.  By Fact~\ref{why-open}, $\tau$ is an interpretable equivalence
  relation on $G$ (but not necessarily Hausdorff).  For each $a \in
  G$, the composition
  \begin{gather*}
    U \hookrightarrow G \times U \to G \\
    x \mapsto (a,x) \mapsto a \cdot x
  \end{gather*}
  is an injective open map, i.e., an open embedding.  So we have
  proven the following:
  \begin{claim} \label{property-claim}
    There is an interpretable topology $\tau$ on $G$ such that for any
    $a \in G$, the map
    \begin{align*}
      i_a : (U,\tau_0 \restriction U) & \to (G,\tau) \\
      i_a(x) &= a \cdot x
    \end{align*}
    is an open embedding.
  \end{claim}
  The family of open embeddings $\{i_a\}_{a \in G}$ is jointly
  surjective, and so the property in Claim~\ref{property-claim}
  uniquely determines $\tau$.  The property is invariant under left
  translation, and therefore $\tau$ is invariant under left
  translations.

  We claim $\tau$ is Hausdorff.  Given distinct $x, y \in G$, there is
  some $a \in G$ such that $\{x,y\} \subseteq a \cdot U$, by
  Lemma~\ref{genericky-tool}(\ref{gt2}).  Then $a \cdot U$ is a Hausdorff open
  subset of $(G,\tau)$ containing $x$ and $y$, so $x$ and $y$ are
  separated by neighborhoods.

  By Remark~\ref{open-obvious}, $(U,\tau_0 \restriction U)$ is
  strongly admissible.  By Lemma~\ref{genericky-tool}(\ref{gt1}), finitely many
  left translates of $U$ cover $G$.  Each of these translates is
  homeomorphic to $(U,\tau_0 \restriction U)$, so by
  Lemma~\ref{finite-cover}, $(G,\tau)$ is strongly admissible.

  In summary, $\tau$ is an interpretable topology, $\tau$ is invariant
  under left translation, $\tau$ is Hausdorff, and $\tau$ is strongly
  admissible.  By Lemma~\ref{upgrade}, $\tau$ is a group topology on
  $G$.
\end{proof}

\begin{theorem} \label{adm-group-thm}
  Let $G$ be an interpretable group.
  \begin{enumerate}
  \item If $\tau$ is an interpretable group topology on $G$, then the
    following properties are equivalent:
    \begin{enumerate}
    \item $\tau$ is strongly admissible.
    \item $\tau$ is admissible.
    \item $\tau$ is manifold dominated.
    \item $\tau$ is definably dominated.
    \end{enumerate}
  \item There is a unique interpretable group topology $\tau$
    satisfying these equivalent conditions.
  \end{enumerate}
\end{theorem}
\begin{proof}
  The weakest of the four conditions is definable domination, and the
  strongest is strong admissibility.  Proposition~\ref{construction-2}
  gives a strongly admissible group topology $\tau_0$.  Suppose $\tau$
  is definably dominated.  Then $\tau$ is admissible by
  Lemma~\ref{to-admissible}, and so $\tau = \tau_0$ by
  Lemma~\ref{uniqueness-0}.
\end{proof}

\subsection{Examples}
\begin{definition}
  An \emph{admissible group} is an interpretable group with an
  admissible group topology.
\end{definition}
By Theorem~\ref{adm-group-thm}, every interpretable group becomes an
admissible group in a unique way.
\begin{example}
  If $G$ is a definable group, then Pillay shows that $G$ admits a
  unique definable manifold topology \cite{Pillay-G-in-p}.  Definable
  manifolds are admissible, so Pillay's topology agrees with the
  unique admissible group topology on $G$.
\end{example}
\begin{example}
  The discrete topology on the value group $\Gamma$ is admissible
  (Example~\ref{gamma}).  This is necessarily the unique admissible
  group topology on $\Gamma$.
\end{example}
\begin{remark} \label{local-dim-remark}
  Let $G$ be an interpretable group, topologized with its unique
  admissible topology.  Let $n = \dim(G)$.  By
  Proposition~\ref{local-dim}, there is a point $a \in G$ at which the
  local dimension is $n$.  By translation symmetry, the local
  dimension is $n$ at every point.  As $G$ is locally Euclidean, we
  see that $G$ is locally homeomorphic to an open subset of $\Mm^n$ at
  any point.  As a consequence, the admissible group topology on $G$
  is discrete if and only if $\dim(G) = 0$.
\end{remark}

\subsection{Subgroups, quotients, and homomorphisms}
In this section, we analyze the topological properties of
interpretable homomorphisms.
\begin{proposition} \label{hom-cts}
  Let $f : G \to H$ be an interpretable homomorphism between two
  admissible groups.  Then $f$ is continuous.
\end{proposition}
\begin{proof}
  Similar to the proof of Lemma~\ref{uniqueness-0}.
\end{proof}
\begin{proposition} \label{closed-subgroup}
  Let $G$ be an interpretable group and let $H$ be an interpretable
  subgroup.  Let $\tau_G$ and $\tau_H$ be the admissible topologies in
  $G$ and $H$.
  \begin{enumerate}
  \item \label{cs1} $H$ is $\tau_G$-closed.
  \item \label{cs2} $\tau_H$ is the restriction of $\tau_G$.
  \end{enumerate}
\end{proposition}
\begin{proof}
  \begin{enumerate}
  \item Because $\tau_G$ is invariant under translations, the
    $\tau_G$-frontier $\partial H$ is a union of cosets of $H$.
    Therefore, one of two things happens: $\partial H = \varnothing$
    (meaning that $H$ is $\tau_G$-closed), or $\dim(\partial H) \ge
    \dim(H)$ (contradicting Proposition~\ref{frontier-dimension}).
  \item The restriction $\tau_G \restriction H$ is admissible by
    Proposition~\ref{subspaces}.  The group operation $H \times H \to
    H$ is continuous with respect to $\tau_G \restriction H$, and so
    $(H,\tau_G \restriction H)$ is an admissible group.  By the
    uniqueness of the topology, $\tau_G \restriction H = \tau_H$. \qedhere
  \end{enumerate}
\end{proof}
\begin{corollary} \label{ce-cor}
  If $f : G \to H$ is an injective homomorphism of admissible groups,
  then $f$ is a closed embedding.
\end{corollary}
\begin{proposition} \label{open-subgroup}
  Let $G$ be an admissible group and $H$ be an interpretable subgroup.
  Then $H$ is open in the admissible topology on $G$ if and only if
  $\dim(H) = \dim(G)$.
\end{proposition}
\begin{proof}
  $H$ is open if and only if $H$ has non-empty interior (as a subset
  of $G$).  If $\dim(H) = \dim(G)$, then $H$ has non-empty interior by
  Proposition~\ref{frontier-dimension}.  Conversely, suppose $H$ has
  non-empty interior.  By Remark~\ref{local-dim-remark}, $\dim_x(G) =
  \dim(G)$ for all $x \in G$.  Taking $x$ in the interior of $H$, we
  see $\dim(H) = \dim(G)$.
\end{proof}

\begin{proposition}\label{quotient-prop}
  Let $G$ be an admissible group and let $H$ be an interpretable
  normal subgroup.
  \begin{enumerate}
  \item \label{qp3} The quotient topology on $G/H$ agrees with the unique
    admissible group topology on $G/H$.
  \item The map $G \to G/H$ is an open map.
  \end{enumerate}
\end{proposition}
\begin{proof}
  Let $E$ be the equivalence relation $xEy \iff xH = yH$.
  \begin{claim}
    $E$ is an OQ equivalence relation on $G$ (Definition~\ref{def-oq}).
  \end{claim}
  \begin{claimproof}
    We verify criterion (\ref{cfo4}) of Remark~\ref{criterion-for-oer}.
    Suppose $g_1, g_2 \in G$ are $E$-equivalent, and $B$ is a
    neighborhood of $g_1$.  Take $h \in H$ such that $g_1 \cdot h =
    g_2$.  Then $B \cdot h$ is a neighborhood of $g_2$, and every
    element of $B \cdot h$ is $E$-equivalent to an element of $B$.
  \end{claimproof}
  Consider $G/H$ with the quotient topology $\tau$.  By the claim, $G
  \to G/H$ is an open map.  By Fact~\ref{why-open}, $\tau$ is
  interpretable.  It remains to show that $\tau$ is the admissible
  group topology on $G/H$.
  \begin{claim}
    $\tau$ is a group topology.
  \end{claim}
  \begin{claimproof}
    We claim $(x,y) \mapsto x \cdot y^{-1}$ is continuous with respect
    to $\tau$.  Fix $a,b \in G/H$.  Take lifts $\tilde{a}, \tilde{b}
    \in G$.  Let $U$ be a neighborhood of $a \cdot b^{-1}$ in $G/H$.
    Let $\pi : G \to G/H$ be the quotient map.  Then $\pi^{-1}(U)$ is
    an open neighborhood of $\tilde{a} \cdot \tilde{b}^{-1}$ in $G$.
    By continuity of the group operations on $G$, there are open
    neighborhoods $W \ni \tilde{a}$ and $V \ni \tilde{b}$ in $G$ such
    that $W \cdot V^{-1} \subseteq \pi^{-1}(U)$, in the sense that
    \begin{equation*}
      x \in W, ~ y \in V \implies x \cdot y^{-1} \in \pi^{-1}(U).
    \end{equation*}
    As $\pi$ is an open map, $\pi(W)$ and $\pi(V)$ are open
    neighborhoods of $a$ and $b$.  Then $\pi(W) \cdot \pi(V)^{-1} =
    \pi(V \cdot W^{-1}) \subseteq \pi(\pi^{-1}(U)) = U$.  This shows
    continuity at $(a,b)$.
  \end{claimproof}
  By Proposition~\ref{closed-subgroup}(\ref{cs1}), $H$ is closed,
  which implies $\{1\} \subseteq G/H$ is closed in $\tau$ by
  definition of the quotient topology.  As $\tau$ is a group topology,
  it follows that $\tau$ is Hausdorff.  Then the open map $G \to G/H$
  shows that $\tau$ is definably dominated, by
  Lemma~\ref{dom-trans}(\ref{dt1}).  Finally, $\tau$ is the admissible
  group topology on $G/H$ by Theorem~\ref{adm-group-thm}.
\end{proof}

\begin{corollary} \label{o-cor}
  Let $f : G \to H$ be a surjective interpretable homomorphism between
  two admissible groups.  Then $f$ is an open map.
\end{corollary}

\section{Definable compactness in strongly admissible spaces} \label{def-com-sec}
Recall definable compactness from
Definition~\ref{fornasiero-definition}.  Let $X$ be an interpretable
topological space in a $p$-adically closed field $M$ with value group
$\Gamma$.
\begin{definition}{{\cite[Definition~2.6]{johnson-yao}}}
  A \emph{$\Gamma$-exhaustion} of $X$ is an interpretable family
  $\{X_\gamma\}_{\gamma \in \Gamma}$, such that
  \begin{itemize}
  \item Each $X_\gamma$ is a definably compact, clopen subset of $X$.
  \item If $\gamma \le \gamma'$, then $X_\gamma \subseteq
    X_{\gamma'}$.
  \item $X = \bigcup_{\gamma \in \Gamma} X_\gamma$.
  \end{itemize}
\end{definition}
\begin{proposition} \label{gx-purpose}
  Let $\{X_\gamma\}_{\gamma \in \Gamma}$ be a $\Gamma$-exhaustion of
  an interpretable topological space $X$.  Then $X$ is definably
  compact if and only if $X = X_\gamma$ for some $\gamma \in \Gamma$.
\end{proposition}
\begin{proof}
  Each $X_\gamma$ is definably compact, so if $X = X_\gamma$ then $X$
  is definably compact.  Conversely, suppose $X$ is definably compact.
  Then the family $\{X \setminus X_\gamma\}_{\gamma \in \Gamma}$ is a
  downward-directed family of closed sets with empty intersection.  By
  definable compactness, some $X \setminus X_\gamma$ must vanish.
\end{proof}
\begin{lemma} \phantomsection \label{zhang-yao}
  \begin{enumerate}
  \item Let $f : X \to Y$ be an interpretable surjective
    open map between Hausdorff interpretable topological spaces.
    Suppose $\{X_\gamma\}_{\gamma \in \Gamma}$ is a
    $\Gamma$-exhaustion of $X$.  Let $Y_\gamma = f(X_\gamma)$ for each
    $\gamma$.  Then $\{Y_\gamma\}_{\gamma \in \Gamma}$ is a
    $\Gamma$-exhaustion of $Y$.
  \item \label{zy2} If $X$ is manifold dominated, then $X$ has a
    $\Gamma$-exhaustion.
  \end{enumerate}
\end{lemma}
\begin{proof}
  \begin{enumerate}
  \item Each $Y_\gamma$ is open because $f$ is an open map.  Each
    $Y_\gamma$ is definably compact as the image of a definably
    compact set under a definable map.  Each $Y_\gamma$ is closed
    because $Y$ is Hausdorff.  We have $\bigcup_\gamma Y_\gamma =
    \bigcup_\gamma f(X_\gamma) = f(X) = Y$, because $f$ is surjective.
    It is clear that if $\gamma \le \gamma'$, then $Y_{\gamma} =
    f(X_{\gamma}) \subseteq f(X_{\gamma'}) = Y_{\gamma'}$.
  \item Take a definable manifold $U$ and an interpretable surjective
    open map $f : U \to X$.  By \cite[Proposition~2.8]{johnson-yao}, $U$
    admits a $\Gamma$-exhaustion.  By the first point, we can push
    this forward to a $\Gamma$-exhaustion on $X$. \qedhere
  \end{enumerate}
\end{proof}
The idea of Lemma~\ref{zhang-yao} came out of discusisons with
Ningyuan Yao and Zhentao Zhang.
\begin{remark}
  Say that an interpretable topological space $X$ is ``locally
  definably compact'' if for every $p \in X$, there is a definably
  compact subspace $Y \subseteq X$ such that $p$ is in the interior of
  $Y$.  Lemma~\ref{zhang-yao} shows that any manifold dominated space
  $X$ is locally definably compact.  Indeed, take a
  $\Gamma$-exhaustion $\{X_\gamma\}_{\gamma \in \Gamma}$.  If $p \in
  X$, then there is $\gamma \in \Gamma$ such that $p \in X_\gamma$,
  and we can take $Y = X_\gamma$.
\end{remark}
\begin{example} \label{impl-2}
  Let $D \subseteq \Mm^2$ be the definable set $\{(x,y) \in \Mm^2 : x
  = 0 ~ \vee ~ y \ne 0\}$ from Example~\ref{impl-1}.  Using
  Fact~\ref{dc-fact-pcf}(\ref{dfp1}), one can see that if $N$ is a
  neighborhood of $(0,0)$ in $D$, and $\overline{N}$ is the closure of
  $(0,0)$ in $D$, then $\overline{N}$ is not definably compact.  In
  other words, local definable compactness fails at $(0,0) \in D$.
  Therefore $D$ is not manifold dominated, and not strongly admissible
  (though it is admissible, trivially).
\end{example}

\begin{theorem} \label{definability-annoy}
  Definable compactness is a definable property, on families of
  strongly admissible spaces, and families of interpretable groups.
  More precisely:
  \begin{enumerate}
  \item \label{da1} If $\{(X_i,\tau_i)\}_{i \in I}$ is an interpretable family of strongly
    admissible spaces, then the set $\{i \in I : (X_i,\tau_i) \text{ is
      definably compact}\}$ is an interpretable subset of $I$.
  \item \label{da2} If $\{G_i\}_{i \in I}$ is an interpretable family of groups,
    then the set \[\{i \in I : G_i \text{ is definably compact with
      respect to the admissible group topology}\}\] is an interpretable
    subset of $I$.
  \end{enumerate}
\end{theorem}
This follows formally from Proposition~\ref{gx-purpose} and
Lemma~\ref{zhang-yao}, using the method of
\cite[Proposition~4.1]{johnson-fsg}.  However, we can give a cleaner
proof, using the following lemma.
\begin{lemma} \label{trickery}
  Let $X$ be strongly admissible and definably compact.  Then there is
  a closed and bounded definable set $D \subseteq \Mm^n$ and a
  continuous interpretable surjection $f : D \to X$.
\end{lemma}
\begin{proof}
  Take $Y$ a definable manifold and $g : Y \to X$ an interpretable
  surjective open map.  Take $U_1, \ldots, U_m \subseteq Y$
  an open cover and interpretable open embeddings $h_i : U_i \to
  \Mm^{n_i}$.  Let $\{U_{i,\gamma}\}_{\gamma \in \Gamma}$ be a
  $\Gamma$-exhaustion of $U_i$ for each $i$.  Then $\left\{\bigcup_{i
    = 1}^m U_{i,\gamma} \right\}_{\gamma \in \Gamma}$ is a
  $\Gamma$-exhaustion of $Y$, and $\left\{ f\left( \bigcup_{i = 1}^m
  U_{i,\gamma} \right) \right\}_{\gamma \in \Gamma}$ is a
  $\Gamma$-exhaustion of $X$ by Lemma~\ref{zhang-yao}.  As $X$ is
  definably compact, there is some $\gamma_0 \in \Gamma$ such that $f
  \left( \bigcup_{i = 1}^m U_{i,\gamma_0} \right) = X$ by
  Proposition~\ref{gx-purpose}.

  Each set $U_{i,\gamma_0}$ is homeomorphic to a definable subset of
  $\Mm^{n_i}$, via the embeddings $g_i$.  Then the disjoint union
  $U_{1,\gamma_0} \sqcup \cdots \sqcup U_{m,\gamma_0}$ is homeomorphic
  to a definable set $D \subseteq \Mm^n$ for sufficiently large $n$.
  The natural map
  \begin{equation*}
    U_{1,\gamma_0} \sqcup \cdots \sqcup U_{m,\gamma_0} \to \bigcup_{i
      = 1}^m U_{i,\gamma_0} \hookrightarrow Y \to X
  \end{equation*}
  is surjective by choice of $\gamma_0$.  So there is a continuous
  surjection $D \to X$.  Finally, $D$ is homeomorphic to the disjoint
  union $U_{1,\gamma_0} \sqcup \cdots \sqcup U_{m,\gamma_0}$, which is
  definably compact, and so $D$ is closed and bounded by
  Fact~\ref{dc-fact-pcf}(\ref{dfp1}).
\end{proof}
Using this we can prove Theorem~\ref{definability-annoy}
\begin{proof}[Proof (of Theorem~\ref{definability-annoy})]
  \begin{enumerate}
  \item Let $S$ be the set of $i \in I$ such that $(X_i,\tau_i)$ is
    definably compact.  It suffices to show that $S$ and $I \setminus
    S$ are $\vee$-definable, i.e., small unions of $M_0$-definable
    sets.  By Lemma~\ref{trickery}, $i \in S$ if and only if there is
    a definable set $D$ and a surjective continuous interpretable map
    $D \to X_i$.  It is easy to see that the set of such $i$ is
    $\vee$-definable.  Meanwhile, by our definition of definable
    compactness, $i \notin S$ if and only if there is a
    downward-directed interpretable family $\{F_j\}_{j \in J}$ of
    non-empty closed subsets of $X_i$ such that $\bigcap_{j \in J} F_j
    = \varnothing$.  Again, the set of such $i$ is $\vee$-definable.
  \item Let $S$ be the set of $i \in I$ such that $G_i$ is definably
    compact, with respect to the unique admissible group topology on
    $G_i$.  Then $i \in S$ (resp. $i \notin S$) if and only if there
    is an interpretable topology $\tau$ on $G_i$, a definable set $D
    \subseteq \Mm^n$, and an interpretable function $f : D \to G_i$
    such that
    \begin{itemize}
    \item $\tau$ is a Hausdorff group topology on $G_i$
    \item $f$ is a surjective open map.
    \item $(G_i,\tau)$ is definably compact (resp. not definably compact).
    \end{itemize}
    Again, these conditions are easily seen to be $\vee$-definable,
    using part (\ref{da1}) for the third point.  Thus $S$ and its complement
    are both $\vee$-definable. \qedhere
  \end{enumerate}
\end{proof}
Using different methods, Pablo And\'ujar Guerrero and the author have
shown that definable compactness is a definable property in \emph{any}
interpretable family of topological spaces
\cite[Theorem~8.16]{andujar-johnson}.  In other words,
Theorem~\ref{definability-annoy}(\ref{da1}) holds without the
assumption of strong admissibility.

\subsection{Definable compactness in $\Qq_p$}
Fix a copy of $\Qq_p$ embedded into the monster $\Mm \models
p\mathrm{CF}$.  If $X$ is a $\Qq_p$-interpretable topological space,
then $X(\Qq_p)$ is naturally a topological space.
\begin{warning}
  $X(\Qq_p)$ is usually not a subspace of $X(\Mm)$.  This is analogous
  to how if $M$ is a highly saturated elementary extension of $\Rr$,
  then $\Rr$ with the order topology is not a subspace of $M$ with the
  order topology.  In fact, if we start with $(M,\le)$ and take the
  induced subspace topology on $\Rr$, we get the discrete topology.
\end{warning}
\begin{proposition} \label{actually-compact}
  Let $X$ be a strongly admissible topological space in $\Qq_p$.  Then
  $X$ is definably compact if and only if $X$ is compact.
\end{proposition}
\begin{proof}
  If $X$ is compact, then $X$ is definably compact by
  Fact~\ref{dc-fact}(\ref{df1}).  Conversely, suppose $X$ is definably
  compact.  By Lemma~\ref{trickery}, there is a closed bounded
  definable set $D \subseteq \Mm^n$ and an interpretable continuous
  surjection $f : D \to X$.  As $\Qq_p \preceq \Mm$ we can take $D$
  and $f$ to be defined over $\Qq_p$.  Then $f : D(\Qq_p) \to
  X(\Qq_p)$ is a continuous surjection and $D(\Qq_p)$ is compact, so
  $X(\Qq_p)$ is compact.
\end{proof}
\begin{proposition} \label{eliminator}
  Let $X$ be $\Qq_p$-interpretable strongly admissible topological
  space.  If $X$ is definably compact, then $X$ is a definable
  manifold.  In particular, there is a $\Qq_p$-interpretable
  set-theoretic bijection between $X$ and a definable set.
\end{proposition}
\begin{proof}
  For each point $a \in X(\Qq_p)$, there is an open neighborhood $U_a$
  and an open embedding $f_a : U_a \to \Mm^{n_a}$ where $n_a$ is the
  local dimension of $X$ at $a$.  Because $\Qq_p \preceq \Mm$, we can
  take $f_a$ and $U_a$ to be $\Qq_p$-interpretable.  Then $U_a(\Qq_p)$
  is an open subset of $X(\Qq_p)$ containing $a$.  By
  Proposition~\ref{actually-compact}, $X(\Qq_p)$ is compact, and so
  there are finitely many points $a_1, \ldots, a_n$ such that
  $X(\Qq_p) = \bigcup_{i = 1}^n U_{a_i}(\Qq_p)$.  Then $X = \bigcup_{i
    = 1}^n U_{a_i}$ because $\Qq_p \preceq \Mm$.  The sets $U_{a_i}$
  and maps $f_{a_i} : U_{a_i} \to \Mm^{n_{a_i}}$ witness that $X$ is a
  definable manifold.
\end{proof}

\section{Interpretable groups and \textit{fsg}} \label{fsg-sec}
Recall the definition of \textit{fsg} (finitely satisfiable generics)
from Section~\ref{intro-fsg}.
\begin{theorem} \label{fsg-char}
  Let $G$ be an interpretable group in a $p$-adically closed field.
  Then $G$ has finitely satisfiable generics (\textit{fsg}) if and
  only if $G$ is definably compact with respect to the admissible
  group topology.
\end{theorem}
\begin{proof}
   The arguments from \cite[Sections 3, 5--6]{johnson-fsg} work almost
   verbatim, given everything we have proved so far.  The existence of
   $\Gamma$-exhaustions, used in \cite[Section 3]{johnson-fsg}, is
   handled by Lemma~\ref{zhang-yao}(\ref{zy2}).  We don't need to redo the
   arguments of \cite[Section 5]{johnson-fsg}---if $G$ is a definably
   compact $\Qq_p$-interpretable group, then $G$ \emph{is already
   definable} by Proposition~\ref{eliminator}, so we can directly
   apply \cite[Proposition~5.1]{johnson-fsg}.\footnote{In particular,
   we don't need to worry about Haar measurability of interpretable
   subsets of $G$ because $G$ is definable.}  The arguments of
   \cite[Section 6]{johnson-fsg} go through, changing ``definable'' to
   ``interpretable'' everywhere.
\end{proof}

\begin{corollary}\label{fsg-def}
  If $\{G_a\}_{a \in Y}$ is an interpretable family of interpretable
  groups, then the set
  \begin{equation*}
    \{a \in Y : G_a \text{ has \textit{fsg}}\}
  \end{equation*}
  is interpretable.
\end{corollary}
\begin{proof}
  Theorems~\ref{fsg-char} and \ref{definability-annoy}.
\end{proof}

\begin{corollary} \label{fsg-int-qp}
  Let $G$ be an interpretable group in $\Qq_p$.  If $G$ has
  \textit{fsg}, then $G$ is interpretably isomorphic to a definable
  group.
\end{corollary}
\begin{proof}
  Proposition~\ref{eliminator}.
\end{proof}

\section{Zero-dimensional groups and sets} \label{0-dim}
In Remark~\ref{local-dim-remark}, we saw that an admissible group $G$
is discrete iff it is zero-dimensional.  This holds more generally for
admissible spaces, by the following two propositions:
\begin{proposition}
  If $X$ is an admissible topological space and $X$ is discrete, then
  $\dim(X) \le 0$.
\end{proposition}
\begin{proof}
  Immediate by considering local dimension
  (Proposition~\ref{local-dim}).
\end{proof}
Conversely, the discrete topology is admissible on zero-dimensional
sets:
\begin{proposition} \label{funny-converse}
  Let $X$ be a zero-dimensional interpretable set.  The discrete
  topology is strongly admissible, and is the only admissible topology
  on $X$.
\end{proposition}
\begin{proof}
  We claim any admissible topology $\tau$ on $X$ is discrete.  Indeed,
  Proposition~\ref{mostly-euclidean} gives a large open subspace $X'
  \subseteq X$ that is locally Euclidean.  Then $\dim(X \setminus X')
  < \dim(X) = 0$, so $X \setminus X' = \varnothing$ and $X' = X$.
  Zero-dimensional locally Euclidean spaces are discrete.

  Meanwhile, Theorem~\ref{construction-1} shows that there is
  \emph{some} strongly admissible topology $\tau$ on $X$.  By the
  previous paragraph, $\tau$ must be the discrete topology.
\end{proof}
\begin{definition} \label{psf-def}
  An interpretable set $S$ is \emph{pseudofinite} if $\dim(S) = 0$,
  and $S$ with the discrete topology is definably compact.
\end{definition}
(In Proposition~\ref{other-psf}, we will see that the requirement
$\dim(S) = 0$ is redundant.)
\begin{proposition} \label{stupid}
  If an interpretable set $S$ is finite, then $S$ is pseudofinite.
\end{proposition}
\begin{proof}
  Suppose $S$ is finite.  Then $\dim(S) = 0$ by
  Proposition~\ref{dim-of-sets}(\ref{dos2}).  The discrete topology on $S$ is
  compact, hence definably compact by Fact~\ref{dc-fact}(\ref{df1}).
\end{proof}
\begin{proposition} \phantomsection \label{psf-prop}
  \begin{enumerate}
  \item \label{psp1} Let $S$ be a $\Qq_p$-interpretable set.  Then $S$ is pseudofinite
    iff $S$ is finite.
  \item \label{psp2} If $\{S_a\}_{a \in I}$ is an interpretable family, then $\{a
    \in I : S_a \text{ is pseudofinite}\}$ is interpretable.
  \end{enumerate}
\end{proposition}
\begin{proof}
  \begin{enumerate}
  \item If $S$ is finite, then $S$ is pseudofinite by
    Proposition~\ref{stupid}.  Conversely, suppose $S$ is
    pseudofinite, so $\dim(S) = 0$ and the discrete topology on $S$ is
    definably compact.  By Proposition~\ref{funny-converse}, the
    discrete topology is strongly admissible.  By
    Proposition~\ref{actually-compact}, the discrete topology is
    compact, so $S$ is finite.
  \item Dimension 0 is definable by
    Proposition~\ref{dim-definable-omega}. Assuming dimension 0, the
    discrete topology is strongly admissible by
    Proposition~\ref{funny-converse}, and then definable
    compactness is definable by Theorem~\ref{definability-annoy}(\ref{da1}). \qedhere
  \end{enumerate}
\end{proof}
\begin{remark}
  Proposition~\ref{psf-prop} characterize pseudofiniteness
  uniquely---it is the unique definable property which agrees with
  finiteness over $\Qq_p$.
\end{remark}
\begin{corollary} \label{exists-infty}
  $\Qq_p^{\eq}$ eliminates $\exists^\infty$: for any $L^\eq$-formula
    $\phi(x,y)$, there is a formula $\psi(y)$ such that
    $\phi(\Qq_p,b)$ is infinite if and only if $b \in \psi(\Qq_p)$.
    (But $\Mm^\eq$ does \emph{not} eliminate $\exists^\infty$, of
    course.)
\end{corollary}
This could probably also be seen by the explicit characterization of
imaginaries in \cite[Theorem~1.1]{pcf-ei}.

For 0-dimensional groups, pseudofiniteness is equivalent to \textit{fsg}:
\begin{proposition} \label{p8.8}
  Let $G$ be a 0-dimensional interpretable group.  Then $G$ has
  \textit{fsg} if and only if $G$ is pseudofinite.
\end{proposition}
\begin{proof}
  By Theorem~\ref{fsg-char}, $G$ has \textit{fsg} if and only if the
  admissible group topology on $G$ is definably compact.  By
  Remark~\ref{local-dim-remark} or
  Proposition~\ref{funny-converse}, the admissible group topology
  on $G$ is discrete.
\end{proof}
For example, the value group $\Gamma$ does not have \textit{fsg}, but
the circle group $[0,\gamma) \subseteq \Gamma$ (with addition ``modulo
  $\gamma$'') \emph{does} have \textit{fsg}.

We close by giving some equivalent characterizations of
pseudofiniteness.
\begin{proposition} \label{other-psf}
  Let $S$ be an interpretable set.  The following are equivalent:
  \begin{enumerate}
  \item \label{op1} $S$ is pseudofinite.
  \item \label{op2} $S$ with the discrete topology is definably compact.
  \item \label{op3} If $\mathcal{D} = \{D_a\}_{a \in I}$ is an interpretable
    family of subsets of $S$, and $\mathcal{D}$ is linearly ordered
    under inclusion, then $\mathcal{D}$ has a minimal element.
  \end{enumerate}
\end{proposition}
\begin{proof}
  Consider two additional criteria:
  \begin{enumerate}
    \setcounter{enumi}{3}
  \item \label{op4} If $\mathcal{D} = \{D_a\}_{a \in I}$ is a downwards-directed
    interpretable family of non-empty subsets of $S$, then $\bigcap_{a
      \in I} D_a \ne \varnothing$.
  \item \label{op5} If $\mathcal{D} = \{D_a\}_{a \in I}$ is a linearly ordered
    interpretable family of non-empty subsets of $S$, then $\bigcap_{a
      \in I} D_a \ne \varnothing$.
  \end{enumerate}
  We claim that (\ref{op1})$\implies$(\ref{op2})$\iff$(\ref{op4})$\implies$(\ref{op5}),
  (\ref{op1})$\implies$(\ref{op3})$\implies$(\ref{op5}), and (\ref{op5})$\implies$(\ref{op1}).  First of all,
  (\ref{op1})$\implies$(\ref{op2}) by our definition of ``pseudofinite,'' and
  (\ref{op2})$\iff$(\ref{op4}) by definition of definable compactness.  The
  implications (\ref{op4})$\implies$(\ref{op5}) and (\ref{op3})$\implies$(\ref{op5}) are trivial.

  If (\ref{op1})$\centernot \implies$(\ref{op3}), then there is a pseudofinite interpretable
  set $S$ and an interpretable chain $\mathcal{D}$ of subsets of $S$
  with no minimum.  Because $\Qq_p \preceq \Mm$ and pseudofiniteness
  is a definable property (Proposition~\ref{psf-prop}(\ref{psp2})), we can
  assume $S$ and $\mathcal{D}$ are interpretable over $\Qq_p$.  Then
  $S$ is actually finite (Proposition~\ref{psf-prop}(\ref{psp1})) and so
  $\mathcal{D}$ must have a minimum, a contradiction.

  It remains to show (\ref{op5})$\implies$(\ref{op1}).  We prove $\neg(\ref{op5})$ assuming
  $\neg(\ref{op1})$.  There are two cases, depending on why (\ref{op1}) fails:
  \begin{description}
  \item[$\dim(S) > 0$:] Then Corollary~\ref{container} shows that
    there is a ball $B$ in $\Mm^n$ for $n = \dim(S)$, and an
    interpretable injection $B \hookrightarrow S$.  Without loss of
    generality, $B$ is $\Oo^n$, where $\Oo$ is the valuation ring on
    $\Mm$.  Then $B$ contains a linearly ordered definable family of
    non-empty subsets with empty intersection, namely the family of
    sets $B(\gamma)^n \setminus \{\bar{0}\}$, where $B(\gamma)
    \subseteq \Mm$ is the ball of radius $\gamma$.  Using the
    embedding $B \hookrightarrow S$, we get a similar family in $S$,
    contradicting (\ref{op5}).
  \item[$\dim(S) = 0$] but the discrete topology is not definably
    compact: Then $S$ with the discrete topology is strongly
    admissible by Proposition~\ref{funny-converse}.  By
    Lemma~\ref{zhang-yao} there is a $\Gamma$-exhaustion
    $\{S_\gamma\}_{\gamma \in \Gamma}$ of $S$.  As $S$ is not
    definably compact, $S \setminus S_\gamma \ne \varnothing$ for each
    $\gamma$, by Proposition~\ref{gx-purpose}.  Then the family $\{S
    \setminus S_\gamma\}_{\gamma \in \Gamma}$ contradicts
    (\ref{op5}). \qedhere
  \end{description}
\end{proof}

\subsection{Geometric elimination of imaginaries} \label{geometric-ei}
The $n$th \emph{geometric sort} is the quotient $S_n :=
GL_n(\Mm)/GL_n(\Oo)$.  The theory $\pCF$ has elimination of
imaginaries after adding the geometric sorts to the language
\cite[Theorem~1.1]{pcf-ei}.  Consequently, we may assume any
interpretable set $X$ is a definable subset of some product $\Mm^n
\times \prod_{i = 1}^k S_{m_i}$.  The geometric sorts are
0-dimensional, so by Proposition~\ref{funny-converse} the discrete
topology on $S_n$ is an admissible topology.  Consequently, if we take
the the standard topology on $\Mm$, the discrete topology on each
$S_{m_i}$, the product topology on $\Mm^n \times \prod_{i = 1}^k
S_{m_i}$, and the subspace topology on $X \subseteq \Mm^n \times
\prod_{i = 1}^k S_{m_i}$, we get an admissible topology on $X$, by
Section~\ref{closure-props}.

This suggests the following alternative approach to tame topology on
interpretable sets in $\pCF$.  Work in the language $L_G$ with the
geometric sorts.  Then all interpretable sets are definable, up to
isomorphism.  Endow the home sort $\Mm$ with the usual valuation
topology, and endow each geometric sort $S_n$ with the discrete
topology.  Endow any $L_G$-definable set $D \subseteq \Mm^n \times
\prod_{i = 1}^k S_{m_i}$ with the subspace topology (of the product
topology).  By analogy with the o-minimal ``definable spaces'' in
\cite{PS}, say that a definable topological space $X$ is a
\emph{geometric definable space} if it is covered by finitely many
open sets $U_1, \ldots, U_n$, each of which is definably isomorphic to
an ($L_G$-)definable set.  Note that geometric definable spaces are
admissible.

Geometric definable spaces should be
an adequate framework for tame topology on interpretable sets and
interpretable groups in $\pCF$:
\begin{enumerate}
\item Every interpretable set is $L_G$-definable by elimination of
  imaginaries, and therefore admits the structure of a geometric
  definable space in a trivial way.
\item The topological tameness results of Section~\ref{ss-tame} hold
  for geometric definable spaces.  We can see this from admissibility,
  but there are probably direct proofs.
\item The methods of Section~\ref{sec:ex} or \cite{Pillay-G-in-p}
  presumably show that any interpretable group can be given the
  structure of a geometric definable space.
\end{enumerate}
Geometric definable spaces might provide a simpler alternative to the
admissible spaces of the present paper.  On the other hand, such an
approach is unlikely to generalize beyond $\pCF$.

\section{Further directions} \label{s-fd}
There are several directions for further research.
\subsection{Extensions to $P$-minimal and visceral theories}

Many of the results of this paper may generalize from $\pCF$ to other
$P$-minimal theories---expansions of $\pCF$ in which every unary
definable set is definable in the pure field sort \cite{p-min}.  The
topological tameness results of $\pCF$ like the cell decomposition and
dimension theory are known to generalize to $P$-minimal theories
\cite{p-min,p-minimal-cells}.

More generally, some of the results of this paper may generalize to
\emph{visceral theories} with the exchange property.  Visceral
theories were introduced by Dolich and Goodrick \cite{viscerality}.
Recall the notion of uniformities and uniform spaces from topology.  A
theory is \emph{visceral} if there is a definable uniformity on the
home sort $\Mm^1$ such that a unary definable set $D \subseteq \Mm^1$
is infinite iff it has non-empty interior.  The theory $\pCF$ is
visceral, as are $P$-minimal theories, o-minimal expansions of DOAG,
C-minimal expansions of ACVF, unstable dp-minimal theories of fields,
and many theories of valued fields.  Dolich and Goodrick prove a
number of topological tameness results for visceral theories,
analogous to those that hold in o-minimal and $P$-minimal
theories.\footnote{Some of Dolich and Goodrick's results are dependent
on the technical assumption ``definable finite choice'' (DFC).  In
forthcoming work, I will show that the assumption DFC can generally be
removed from all of Dolich and Goodrick's results, as one would expect
\cite{own-visceral}.}

It seems likely that the results of Sections~\ref{adm-sec-1} and
\ref{adm-sec-2} generalize to visceral theories with the exchange
property.  (There are some subtleties around the proof of
Lemma~\ref{step-3}, but these problems are not insurmountable.)  On
the other hand, local definable compacntess fails to hold in the
visceral setting, so Sections~\ref{def-com-sec}--\ref{0-dim} probably
do not generalize.

\subsection{Zero-dimensional \textit{dfg} groups}
An interpretable group $G$ is said to have \textit{dfg} (definable
$f$-generics) if there is a global definable type $p$ on $G$ with
boundedly many left translates.  In distal theories like $\Qq_p$, the
two properties \textit{fsg} and \textit{dfg} are polar opposites, in
some sense.  For example, if $G$ is an infinite interpretable group,
then $G$ can have a most one of the two properties \textit{fsg} and
\textit{dfg} (essentially by \cite[Proposition~2.27]{pierre-distal}).
If $G$ has dp-rank 1 and is definably amenable, then $G$ satisfies
exactly one of the two properties (essentially by
\cite[Theorem~2.8]{surprise} and \cite[Proposition~8.21]{NIPguide}).
By Proposition~\ref{eliminator}, the \textit{fsg} interpretable groups
over $\Qq_p$ are definable.  This vaguely suggests the following
Conjecture:
\begin{conjecture}
  If $G$ is a $\Qq_p$-interpretable 0-dimensional, definably amenable
  group, then $G$ has \textit{dfg}.
\end{conjecture}
The intuition is that ``zero-dimensional'' is the opposite of
``definable,'' for infinite groups.  For example, the value group
$\Gamma$ has \textit{dfg}.

\subsection{The adjoint action}
Suppose $G$ is interpretable.  By Theorem~\ref{adm-group-thm}, the
unique admissible group topology on $G$ is locally Euclidean.  In
particular, $G$ is a manifold in a weak sense.  Because of generic
differentiability, we can probably endow $G$ with a $C^1$-manifold
structure, and then look at how $G$ acts on the tangent space at $1_G$
by conjugation.  That is, we can look at the adjoint action of $G$.

Say that a group is ``locally abelian'' if there is an open
neighborhood $U \ni 1_G$ on which the group operation is commutative.
If $G$ is locally abelian, then the adjoint action is trivial.  The
converse should hold by reducing to the case where $G$ is defined over
$\Qq_p$ and using properties of $p$-adic Lie groups.  If $G$ is
locally abelian, witnessed by $U \ni 1_G$, then the center of the
centralizer of $U$ is an abelian open subgroup.  Thus, locally abelian
groups have open abelian interpretable subgroups.

For a general interpretable group $G$, the adjoint action gives a
homomorphism $G \to GL_n(\Mm)$, where $n = \dim(G)$.  The kernel
should be a locally abelian group, and the image is definable.  Thus,
every interpretable group should be an extension of a definable group
by a locally abelian group.

This suggests the question: which groups are locally abelian?  Can we
classify them?  Zero-dimensional interpretable groups are locally
abelian, and so are abelian interpretable groups.  Are all locally
abelian interpretable groups built out of 0-dimensional groups and
abelian groups?  If $G$ is locally abelian, is there a \emph{normal}
abelian open subgroup?


\begin{acknowledgment}
  The author was supported by the National Natural Science Foundation
  of China (Grant No.\@ 12101131).  The idea of writing this paper
  arose from discussions with Ningyuan Yao and Zhentao Zhang, who
  convinced me that interpretable groups in $\pCF$ could be
  meaningfully topologized.  Alf Onshuus provided some helpful
  references.  Anand Pillay asked what these results imply for groups
  interpretable in the value group, which inspired
  Section~\ref{0-dim}.
\end{acknowledgment}

\bibliographystyle{alpha} \bibliography{minibib}{}

\end{document}